\documentclass[a4paper,12pt]{article}
\date{}
\usepackage{mathrsfs}
\usepackage[latin1]{inputenc}
\usepackage[english]{babel}
\usepackage{mathrsfs}
\usepackage{amsmath}
\usepackage{amsfonts}
\usepackage{amsthm}

\usepackage{amssymb, graphicx, epsfig}

\usepackage{latexsym}

\newcommand{\radsumma}[2]{\genfrac{}{}{0pt}{}{#1}{#2}}
\newtheorem{thm}{Theorem}[section] 
\newtheorem{Lemma}{Lemma}[section]
\newtheorem*{ass}{(A1)}
\newtheorem*{ass2}{(A2)}
\newtheorem*{ass3}{(A3)}
\theoremstyle{definition}
\newtheorem{rem}{Remark}[section]

\author{Cecilia~Holmgren}
\title{A Weakly 1-Stable Limiting Distribution for the Number of Random Records and Cuttings in Split Trees
}

\begin{document}
\maketitle
\begin{abstract}
\emph{We study the number of random records in an arbitrary split tree (or equivalently, 
the number of random cuttings required to eliminate the tree). We show that a classical limit theorem for convergence of sums of triangular arrays to infinitely divisible distributions can be used to determine the distribution of this number. After normalization 
the distributions are shown to be asymptotically weakly 1-stable. This work is a generalization of our earlier results for the random binary search tree in \cite{holmgren}, which is one specific case of split trees. Other important examples of split trees include $m$-ary search trees, quadtrees, medians of $(2k+1)$-trees, simplex trees, tries and digital search trees.
}\end{abstract}

\section{Introduction}
\subsection {Preliminaries}

We study the number of records in random split trees which were introduced by Devroye \cite{devroye3}. As shown by Janson \cite{jan2}, this number is equivalent (in distribution) to the number of cuts needed to eliminate this type of tree.

Given a rooted tree $T$, let each vertex $v$  have a random value $\lambda{ _v}$
 attached to it, and assume that these values are i.i.d.\ with a continuous distribution. We say that the
 value $\lambda{ _v}$ is a \emph{record} if it is the smallest value in the path from the root to
 $v$. Let $X_v(T)$ denote the (random) number of records. Alternatively one may attach random variables to the edges and let $X_e(T)$ denote the number of edges with record values. 
Only the order relations of the $\lambda{ _v}$'s are important, so the distribution of $\lambda{ _v}$ does not matter, i.e., one can choose any continuous distribution for $\lambda _v$.

The same random variables 
 appear when we consider \emph{cuttings} of the tree 
$T$ as introduced by Meir and Moon \cite{Moon} with the following definition.
 Make a random cut by choosing one vertex [respectively edge] at random. Delete this vertex [respectively edge] so that the tree separates into several parts and keep only
the part containing the root. Continue recursively until the root is cut [respectively only the root is left]. Then the total (random) number of cuts made is $X_v(T)$ [respectively $X_e(T)$]. More precisely, cuttings and records give random variables with the same distribution. 
The proof of this equivalence uses a natural coupling argument as shown in \cite{jan2,jan1}.

In \cite{jan2} the asymptotic distributions for the number of cuts (or the number of records) are found for random trees that can be constructed
as conditioned Galton--Watson trees, e.g., labelled trees and random binary trees. There the proof relies on 
the fact that the method of moments can be used.

For the deterministic (non random) complete binary tree it is, however, not possible to use the method of moments. To deal with this
Janson \cite{jan1} introduced another strategy, which is to approximate $X_v(T)$ by a sum of independent random variables derived from $\lambda{ _v}$, and then apply a classical limit theorem for
triangular arrays, see e.g., \cite[Theorem 15.28]{kallenberg}. We recently showed that Janson's approach could also be applied to the random binary search tree  \cite{holmgren}. 
 
In this paper \emph{we consider all types of (random) split trees} defined by Devroye \cite{devroye3}; the binary search tree that we consider in \cite{holmgren} is one example of such trees. Some other important examples of split trees are $m$-ary search trees, quadtrees, median of $(2k+1)$-trees, simplex trees, tries and digital search trees. The split trees belong to the family of so-called 
$\log{n}$ trees, that are trees with height (maximal depth) $a.a.s.$\ $\mathcal O (\log{n})$. (For the notation $a.a.s.$\, see \cite{jan4}.)  These have similar properties to the deterministic complete binary tree with height $\lfloor\log_{2}{n}\rfloor$ considered in \cite{jan1}. In the complete binary tree (with high probability) most vertices are close to $\lfloor \log_2 n\rfloor$ (the height of the tree). In split trees on the other hand (with high probability) most vertices are close to depth $\sim c \ln n$, where $c$ is a constant (it is natural to use the $e$-logarithm); for the binary search tree that we investigated in \cite{holmgren} this depth is $\sim 2\ln n$ (e.g., \cite{devroye2}).
\emph{Here by the use of renewal theory we extend the methods used in \cite{holmgren} for the specific case of the binary search tree to show that also for split trees in general it is possible to apply a limit theorem, see e.g., \cite[Theorem 15.28]{kallenberg}, for convergence of sums of triangular arrays to infinitely divisible distributions to determine the asymptotic distribution of $X_v(T)$.}


\textbf{The split tree generating algorithm:}  

\begin{figure}[h]
\begin{center}
\includegraphics[scale=0.40]{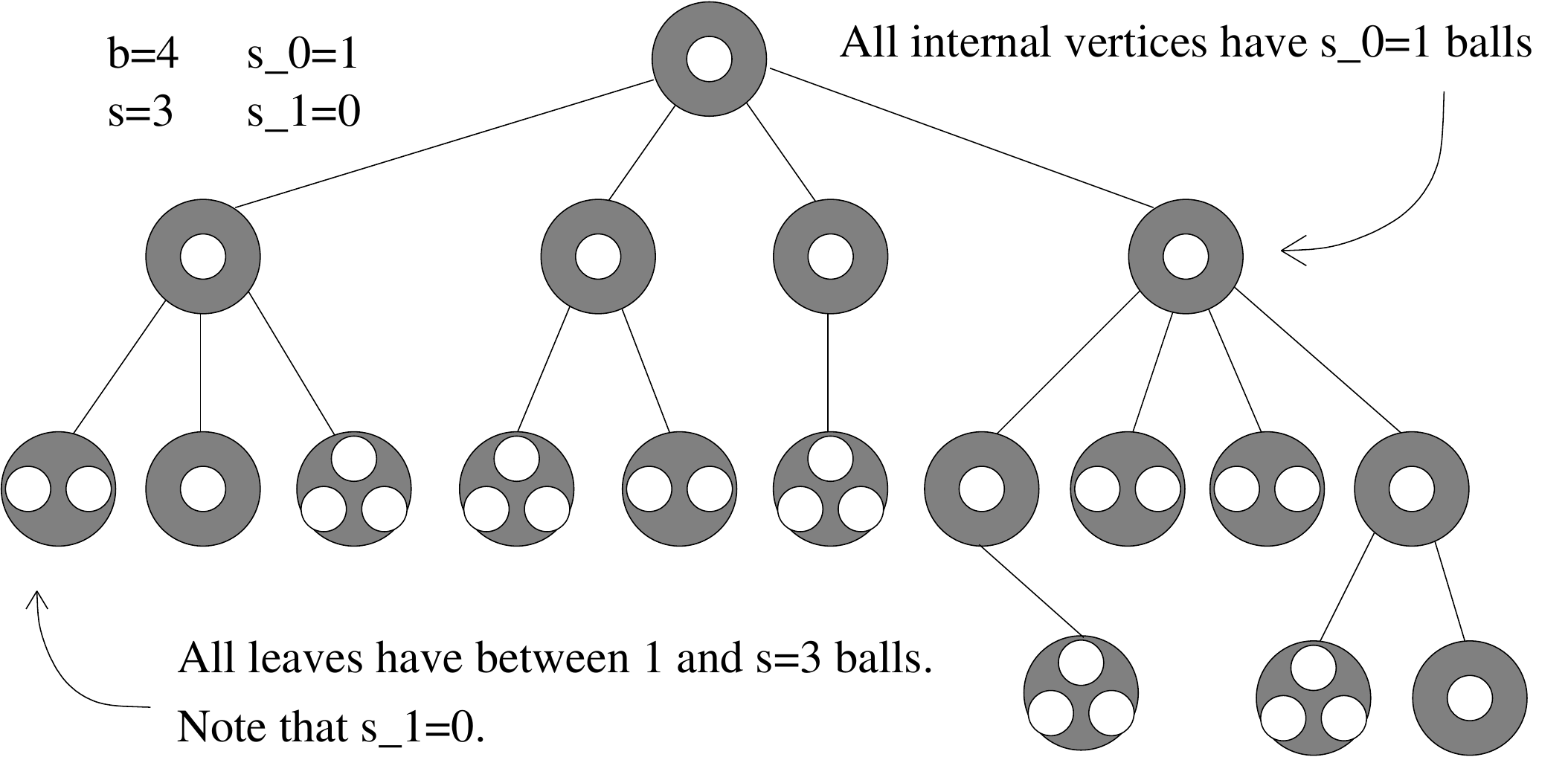}
\end{center}
\caption{\small{\textit{This figure illustrates a split tree with parameters
$b=4,~s=3,~s_0=1$ and $s_1=0$. }}}
\label{boll}
\end{figure}

\begin{figure}[h]
\begin{center}
\includegraphics[scale=0.40]{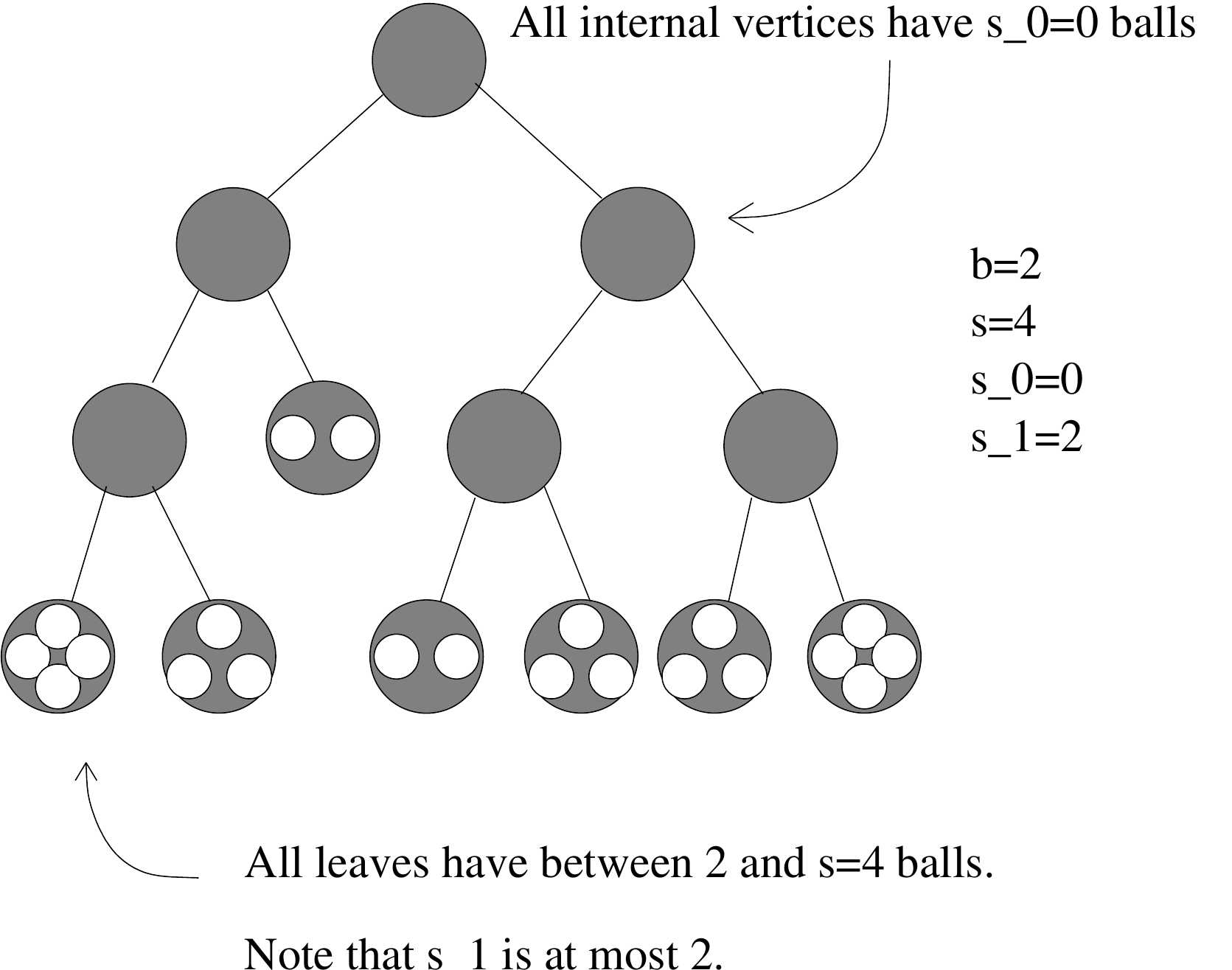}
\end{center}
\caption{\small{\textit{ This figure illustrates a split tree with parameters
$b=2,~s=4,~s_0=0$ and $s_1=2$.}}}
\label{boll2}
\end{figure}

The formal, comprehensive ``split tree generating algorithm'' is as follows with the following introductory notation, see \cite{devroye3} and \cite{holmgren2}.
A split tree is a finite subtree of a skeleton tree $S_b$ (i.e., an infinite rooted tree in which each vertex
 has exactly $b$ children that are numbered $1,2,\dots, b$).  The split tree is constructed recursively by distributing balls one at a time to generate a subset of vertices of $S_b$. We say that the tree has cardinality $n$, if $n$ balls are distributed. There is also a so-called vertex capacity, $s>0$, which means that each node can hold at most $s$ balls. Each vertex $v$ of $S_b$ is given an independent copy of the so-called random split vector $\mathcal{V}=(V_1,V_2 \dots ,V_b)$
 of probabilities, where $\sum_i V_i=1,~~V_i\geq 0$. There are also two other parameters: $s_0,s_1$ (related to the parameter $s$) that occur in the algorithm below; see Figure \ref{boll} and Figure \ref{boll2}, where two examples of split trees are illustrated.
Let $n_v$ denote the total number of balls that the vertices 
in the subtree rooted at vertex $v$ hold together, and $C_v$ be the number of balls that are held by $v$ itself. We say that a vertex $v$ is a leaf in a split tree if the node itself holds at least one ball but no descendants of $v$ hold any balls. An equivalent definition of a leaf is to say that $v$ is a leaf if and only if $C_v=n_v>0$. A vertex $v\in S_b$ is included in the split tree if, and only if, $n_v>0$; if  $n_v=0$, the vertex $v$ is not included and it is called useless. 

Below there is a description of the algorithm which determines how the $n$ balls are distributed over the vertices. 
Initially there are no balls, i.e., $C_v=0$ for each vertex $v$. 
Choose an independent copy  $\mathcal{V}_v$ of  $\mathcal{V}$ for every vertex $v \in S_b$. Add balls one by one to the root by the following recursive procedure for adding a ball to the subtree rooted at $v$.

\begin{enumerate}

\item\label{first} If $v$ is not a leaf, choose child $i$ with probability $V_i$
and recursively add the ball to the subtree rooted at child $i$, by the rules given in steps \ref{first}, \ref{second} and \ref{last}.

\item\label{second}  If $v$ is a leaf and $C_v=n_v<s$,  then add the ball to $v$ and stop. 
Thus, $C_v$ and $n_v$ increase by 1.

\item \label{last} If $v$ is a leaf and  $C_v=n_v=s$, the ball cannot be placed at $v$ since it 
is occupied by the maximal number of balls it can hold. In this case, let $n_v=s+1$ and $C_v=s_0$, by placing $s_0\leq s$ 
randomly chosen balls at $v$ and $s+1-s_0$ balls at its children. This is done by first giving $s_1$ randomly chosen balls
to each of the $b$ children. The remaining $s+1-s_0 -bs_1$ balls are placed by choosing a child for each ball independently 
according to the probability vector $\mathcal{V}_v=(V_1,V_2 \dots ,V_b)$, and then using the algorithm described in steps 
\ref{first}, \ref{second} and \ref{last} applied to the subtree rooted at the selected child.

\end{enumerate}

From \ref{last} it follows that the integers  $s_0$ and $s_1$  have to satisfy the inequality 
\begin{align*}0\leq s_0\leq s,~0\leq bs_1\leq s+1-s_0.
\end{align*}

We can assume that the components $V_i$ of the split vector $\mathcal{V}$ are identically 
distributed. If this were not the case they can anyway be made identically distributed by using a random permutation as explained in \cite{devroye3}. Let $V$ be a random variable with this distribution.
This gives (because $\sum_i V_i=1$) that $\mathbf{E}(V)=\frac{1}{b}$.
We use the notation $T^n$ to denote a split tree with $n$ balls. However, note that even given the fact that the split tree has $n$ balls, the number of nodes $N$, is still a random number.
The only parameters that are important in this work (and in general these parameters are the important ones for most results concerning split trees) are the cardinality $n$, the branch factor $b$ and the split vector $\mathcal{V}$; this is illustrated in Section \ref{Subtrees}.
In a binary search tree $b=2$, the split vector $\mathcal{V}=(V_1,V_2)$ is distributed as $(U,1-U)$ where $U$ is a uniform $U(0,1)$ random variable. For the binary search tree the number of balls $n$ is the same number as the number of vertices $N$; this is not true for split trees in general.

\subsection{Some Important Facts and Results for Split Trees}


 \subsubsection{Results Concerning Depth Analysis}\label{strong}

 In \cite[Theorem 1]{devroye3} Devroye presents a strong law and a central limit law for the depth $D_n$ of the last inserted ball in a split tree with $n$ balls and split vector $\mathcal{V}$. 
Recall that $V$ is distributed as the identically distributed components in the split vector. Let
\begin{align}\label{splitdef}\mu &:=b\mathbf{E}\Big(-V \ln({V})\Big),\nonumber\\
\sigma ^2&:=b\mathbf{E}\Big(V\ln^2 V\Big)-\mu ^2.
\end{align}
If $\mathbf{P}(V=1)=0$ and $\mathbf{P}(V=0)<1$, then
\begin{align}\label{splitdepth}\frac{D_n}{\ln n}\stackrel{d}\rightarrow \mu^{-1},
\end{align}and
\begin{align}\label{jul}\frac{\mathbf{E}(D_n)}{\ln n}\rightarrow \mu^{-1}.
\end{align}
Furthermore, if $\sigma >0$,  then \begin{align}\label{splitdepth:2}\frac{D_n-\mu^{-1} \ln n}{\sqrt{\sigma ^2\mu^{-3}\ln n}}\stackrel{d}\rightarrow N(0,1),
\end{align}
where $N(0,1)$ denotes the standard Normal distribution and $\stackrel{d}\rightarrow$ denotes convergence in distribution. Assuming that $\sigma >0$ is equivalent to  assuming that $V$ is not monoatomic, i.e., it is not the case that $V=\frac{1}{b}$.

Let $D_k$ be the depth of the $k$-th ball.
 In \cite[Theorem 2.3]{holmgren2} by using the same assumptions for $V$ as Devroye uses for proving  (\ref{splitdepth}) and (\ref{splitdepth:2}) we also show results concerning the variance of depths, i.e., for all $\frac{n}{\ln n}\leq k \leq  n$,
\begin{align*}\dfrac{\mathbf{Var}(D_k)}{\ln n}\rightarrow \sigma^2\mu^{-3}.
 \end{align*}

 \subsubsection{Results Concerning the Number of Nodes}\label{nodes}

\begin{ass}\label{assumption2}   In this work we assume as in Section \ref{strong} that $\mathbf{P}(V=1)=0$, and as in \cite{holmgren2} for simplicity we also assume  that $\mathbf{P}(V=0)=0$ and that $-\ln {V}$ is non-lattice. 
 \end{ass}
The non-lattice assumption we do since we use renewal theory for sums depending on the distribution of $-\ln {V}$, and in renewal theory it often becomes necessary to distinguish between lattice and non-lattice distributions. Tries and digital search trees are special forms of split trees with a random permutation of deterministic components $(p_1,p_2,\dots,p_b)$ and therefore not as random as many other examples. Of the common split trees only for some special cases of tries and digital search trees (e.g., the symmetric ones $p_1=p_2=\dots=p_b=\frac{1}{b}$) does $-\ln {V}$ have a lattice distribution.
By assuming that (A1) holds we show in \cite[Theorem 2.1]{holmgren2} that there is a constant $\alpha$ depending on the type of split tree such that for the random number of nodes $N$ we have that
\begin{align}\label{vertexball}\mathbf{E}(N)=\alpha n+o(n);
\end{align} and  
\begin{align}\label{vertexball2}\mathbf{Var}(N)=o(n^{2}).
\end{align}

Let $d(v)$ denote the depth of a node. In \cite[Theorem 2.2]{holmgren2} we show that the expected number of nodes in a tree with $n$ balls, where  $d(v)\leq \mu^{-1}\ln{n}-\ln^{0.5+\epsilon}{n}$ or $d(v)\geq \mu^{-1}\ln{n}+\ln^{0.5+\epsilon}{n}$, for some arbitrary $\epsilon>0$, is  $\mathcal O \Big(\frac {n}{\ln^{k}{n}}\Big)$, for any constant $k$. In this paper we use in particular that this is $O \Big(\frac {n}{\ln^{3}{n}}\Big)$. In \cite[Remark 4.3]{holmgren2} we also note that for any constant $r$ there is a constant $C>0$ such that the expected number of nodes with $d(v)\geq C \ln n$  is $\mathcal{O}\Big(\frac{1}{n^r}\Big)$, hence, we can bound the number of vertices with "large" depths with very small error terms.

 \subsubsection{Results Concerning the Total Path Length}\label{path}
In the present study we consider the ``total path length'' of a tree $T$ as the sum of all depths of the vertices in $T$ (distances to the root).  Since the split tree is a random tree the total path length is a random variable, which we denote by $\Upsilon(T)$.
However, a more natural definition of the total path length is probably the sum of all depths of balls in $T$, which we denote by $\Psi(T)$. 

From (\ref{jul}) it follows that  
 \begin{align}\label{exptotal1}
 \mathbf{E}\Big(\Psi{(T^{n})}\Big)=\mu^{-1} n\ln n+nq(n),
\end{align}
where
$q(n)=o(\ln n)$ is a function that depends on the type of split tree.
By using (\ref{jul}) and  (\ref{vertexball}) we easily show in \cite{holmgren2} that 
 \begin{align}\label{exptotal2}
 \mathbf{E}\Big(\Upsilon{(T^{n})}\Big)=\mu^{-1}\alpha n\ln n+nr(n),
\end{align}
where $\alpha$ is the constant that occurs in (\ref{vertexball}) and 
$r(n)=o(\ln n)$ is a function that depends on the type of split tree.

\begin{ass2}\label{christmas2}
Assume that the functions $q(n)$ in (\ref{exptotal1}) converges to some  constant $\varsigma$.
\end{ass2} 
In \cite{Neininger} there is an analogous assumption.
Examples of split trees where it is shown that $q(n)$ converges to a constant  are binary search trees (e.g.\ \cite{jan3}), random $m$-ary search trees \cite{Mahmoud}, quad trees \cite{Neininger} and the random median of a $(2k+1)$-tree \cite{Roesler}, tries and Patricia tries \cite{bourdon}. 

\begin{ass3}\label{assumption,1} Assume that the result  in (\ref{vertexball}) can be improved 
such that 
\begin{align*}\mathbf{E}(N)=\alpha n+f(n),
\end{align*} where $f(n)=\mathcal{O}\Big(\frac{n}{\ln^{1+\epsilon} n}\Big)$.\end{ass3} 
 Stronger second order terms of the size have previously been shown to hold e.g., for $m$-ary search trees \cite{Mahmoud2}, for these $f(n)$ in assumption (A3) is $o(\sqrt{n})$ when $m\leq 26$ and is
$\mathcal{O}\Big(n^{1-\epsilon}\Big)$ when $m\geq 27$.  Further, as described in Section \ref{nodes} tries are special cases of split trees which are not as random as other types of split trees. Flajolet and Vall\'ee (personal communication) have recently shown that also for most tries (as long as $-\ln V$ is not too close to being lattice) assumption (A3) holds.

In \cite[Theorem 5.1]{holmgren2} by assuming (A2) and (A3) we show that 
$r(n)$ in (\ref{exptotal2}) converges to some constant $\zeta$.
In \cite[Theorem 5.2]{holmgren2} by applying \cite[Theorem 5.1]{holmgren2} we show the following result, which we will apply in the proof of the main theorem below:
 Let  $L=\lfloor \beta \log_b\ln{n}\rfloor$ for some large constant $\beta$, then
\begin{align}\label{TPL}&\sum_{i=1}^{b^L}\frac {\Upsilon{(T_i)}}{\mu^{-2}{\ln^2{n_i}}}=\sum_{i=1}^{b^L}
\frac{\alpha n_i}{\mu^{-1}\ln{n_i}}+\frac{ n \zeta}{\mu^{-2}\ln^2{n}} +o_p\Big(\frac {n}{{\ln^2{n}}}\Big),\end{align} where $\zeta$ is the constant  that $r(n)$ converges to.

\subsection{The Main Theorem}

The main theorem of this study is presented below:

\begin{thm}\label{thm} Let $n\rightarrow\infty$, and suppose that assumptions (A1)--(A3) hold. Then
\begin{align}\label{sats}\left(X_v(T^n)-C_n \right)
\Bigm/\frac{\alpha n}{\mu^{-2}\ln^2{ n}} \stackrel{d}{\rightarrow }-W,\end{align} where \begin{align}C_n:=\frac{\alpha n}{\mu^{-1}\ln{ n}} +\frac{\alpha n\ln\ln{ n}}{\mu^{-1}\ln^2{ n}}-
\frac{\zeta n}{\mu^{-1}\ln^2{ n}},\end{align} for the constant $\zeta$ in (\ref{TPL}), and where $W$ has an infinitely 
divisible distribution. More precisely $W$ has a weakly 1-stable distribution, 
with characteristic function \begin{align}\label{mainthm3}
 \mathbf{E}\Big(e^{itW}\Big)=\exp\Big(-\frac{\mu^{-1}}{2 }\pi|t|+it\Big(C-\mu^{-1}\ln{|t|}\Big)\Big),
\end{align}
where $\mu $ is the constant in (\ref{splitdef}) and $\alpha$ is the constant in (\ref{vertexball}) and $C$ is a constant which is defined in (\ref{amy}) below. The same result holds for $X_e(T^n)$.\end{thm}
\begin{rem}
Even if we only have $\mathbf{E}(N)=\alpha n+o(n)$ as in (\ref{vertexball}) (i.e., ignoring the assumptions  (A2)--(A3) the normalized $X_v(T^n)$ (or $X_e(T^n)$) ought to still converge to a weakly 1-stable distribution with characteristic function as in (\ref{mainthm3}) for some constant $C$. However, in this case $C_n$ in (\ref{sats}) ought to be \begin{multline*}C_n:=2\frac{\mathbf{E}(N)}{\mu^{-1}\ln{ n}}-2\mathbf{E}\bigg(\sum_{d(v)=L}\frac{\mathbf{E}(N_v|n_{v})\ln(\frac{n_{v}}{n})}{\mu^{-1}\ln^2{n}}\bigg)-
\frac{\mathbf{E}(N)L}{\mu^{-1}\ln^2{ n}}\\+\frac{\alpha n\ln\ln{ n}}{\mu^{-1}\ln^2{ n}}-\mathbf{E}\bigg(\sum_{i=1}^{b^L}\frac{\Upsilon(T_i)}{\mu^{-2}{\ln^2{n_i}}}\Big|n_i\bigg),\end{multline*}
where $\Upsilon(T_i)$ is the total path length for the  nodes of the subtrees $T_i$ rooted at depth $L$.
\end{rem}

The class of $\alpha$-stable distributions are included in the larger class of infinitely divisible distributions.
The general formula for the characteristic function of an infinitely divisible distribution is 
\begin{align}\label{characteristic}\exp\left(itb-\frac{a^2t}{2}
+\int_{-\infty}^{\infty}(e^{itx}-1-itx\mathbf{1}[|x|<1])d\nu(x)\right),
\end{align}
for constants $a\geq 0$, $b\in \mathbb R $ and $\nu$ is the so called L\'evy measure. 
The characteristic function in (\ref{characteristic}) of a $1$-stable distribution (i.e., $\alpha=1$) can be simplified to
\begin{align*}\exp\left(i d t-c|t|\Big(1+i\beta \frac{2}{\pi }\mathrm{sign}(t)\ln|t|\Big)\right),
\end{align*}
for constants $c>0$, $\beta\in[-1,1]$ and $d\in \mathbb R $. If the L\'evy measure $\nu$ in (\ref{characteristic}) satisfies  $\frac{d\nu}{dx}=\frac{c_{\pm} }{|x|^{\alpha+1}}$ on $\mathbb R_\pm $, for 
$\alpha\in (0,2)$ and constants $c_{\pm}$ the corresponding infinitely divisible distribution is weakly $\alpha$-stable.
The most well-known 1-stable distribution is the Cauchy distribution. 
However, in contrast to the distribution of $W$  in Theorem \ref{thm} (which is weakly 1-stable), the Cauchy distribution is strictly 1-stable and symmetric.  
The random variable $W$ in Theorem \ref{thm} has support on $(-\infty,\infty)$, and has a heavy tailed distribution. 
As for other random variables with $\alpha$-stable distributions where $\alpha<2$ the  variance of $W$ is infinite. Also since $\alpha\leq 1$ the expected value of $W$ is not defined. 
  For further information about stable distributions, see e.g.,  \cite[Section XVII.3]{feller}.

\begin{rem}\label{rem3}\rm
In the proof of Theorem \ref{thm} we get 
\begin{align}\label{mainthm2} &\mathbf{E}\Big(e^{itW}\Big)=\exp\left(it\Big(C+\mu^{-1}(\gamma -1)\Big)+
\int_{0}^{\infty}(e^{itx}-1-itx\mathbf{1}[x<1])d\nu(x)\right),
\end{align} where $C$ is the constant in (\ref{mainthm3}), $\gamma$ is the Euler constant
 and the L\'evy measure $\nu$ is supported 
on $(0,\infty)$ and has density \[\frac{d\nu}{dx}=\frac{\mu ^{-1}}{x^2}.\] 
Thus,  we see that $W$ has a weakly 1-stable distribution.
The constant $C$ can be expressed as 
\begin{align}\label{amy}C=-\mu^{-1}\ln\mu^{-1}+2\mu^{-1}-\mu^{-2}\sigma^2-\mu^{-1}\gamma-\frac{\sigma^2-\mu^{2}}{2\mu^{2}}, \end{align}
 where $\mu $ and $\sigma ^2$  are the constants in (\ref{splitdef}). We can simplify the expression  in 
(\ref{mainthm2})  to get (\ref{mainthm3}) above. \end{rem}

\begin{rem}\label{rem1} We note in analogy with \cite{jan1} and \cite{holmgren} that most records occur close to the depth where most vertices are, i.e., $\sim \mu ^{-1}\ln{n}$ for split trees. 
Also in analogy with \cite{jan1} and \cite{holmgren}, from Lemma \ref{lem4} and the proof of Theorem \ref{thm3} it follows that most of the random fluctuations of $X_v(T^n)$ can be explained
by the values at depths close to $\ln{\ln{n}}$. \end{rem}

\begin{rem}\label{rem2} 
For random trees \cite{holmgren}, $\mathbf{E}(X_e(T))$=$\mathbf{E}\Big(\sum_{v\neq\sigma }\frac{1}{d(v)}\Big)$ (where $\sigma$  is the root) and $\mathbf{E}(X_v(T))=\mathbf{E}\Big(\sum_{v}\frac{1}{d(v)+1}\Big)$. 
Thus, as we noted for the specific case of the binary search tree \cite[Remark 1.3]{holmgren} also for all other split trees 
\begin{align*}\mathbf{E}(X_e(T^n))-\mathbf{E}(X_v(T^n))=\mathbf{E}\Big(\sum_{v\neq\sigma }\frac{1}{d(v)(d(v)+1)}\Big)-1
\sim C_1 \frac{\alpha n}{\log^2{n}},\end{align*} for some constant $C_1>0$, while there is no similar difference in the limit distribution, see Theorem 1.1 above. 
 As in \cite{holmgren}, this behaviour suggests that it is impossible to use
the method of moments to find the record distribution for split trees as one could do for the conditioned Galton-Watson trees in \cite{jan2}. In \cite{holmgren} we instead  used methods similar to those that Janson used for the complete binary tree in \cite{jan1}. In this paper we generalize the proofs in \cite{holmgren} to consider general split trees.
\end{rem}

\begin{rem}
 Most likely the method that is used here should work for other trees of logarithmic height as well, 
and thus the limiting distribution for these trees should also be infinitely divisible and probably also weakly 1-stable. 
This turns out to be the case for the random recursive tree (that is a logarithmic tree), where the limiting distribution of $X_e(T)$ was recently found to be weakly 1-stable, see \cite[Theorem 1.1]{drmota3} and \cite[Theorem 1.1]{iksanov}. However, the methods used for the recursive tree in \cite{drmota3,iksanov} differ completely from our methods. 
 The advantage with studying split trees compared to the whole class of $\log n$ trees is that there is a common definition that describe all split trees and this is the reason why we only consider these trees in this paper. 
\end{rem}

\subsection{Renewal theory applications for studies of split trees}
\subsubsection{Subtrees}\label{Subtrees}
For the split tree  where the number of balls $n>s$, there are $s_0$ balls in the root and the cardinalities of the 
$b$ subtrees are distributed as $(s_1,\dots, s_1)$ plus a multinomial vector $(n-s_0-bs_1,V_1,\dots, V_b)$.
Thus, conditioning on the random $\mathcal{V}$ -vector that belongs to the root, the subtrees rooted at the children have cardinalities close to $nV_1,\dots, nV_b$.
This is often used in applications of random binary search trees. In particular we used this frequently in \cite{holmgren}.

Conditioning on the split vectors, $n_v$ at depth $d$, is in the stochastic sense bounded from above by
\begin{align}\label{binomial} n_v &\leq \mathrm{Binomial}(n,\prod_{r=1}^{d}W_{r,v})+s_1 d,
\end{align}
and bounded from below by 
\begin{align}\label{binomial,1} n_v &>\mathrm{Binomial}(n,\prod_{r=1}^{d}W_{r,v})-s d,
\end{align}
where $W_{r,v},r\in\{1,\dots, d\}$, are i.i.d.\ random variables given by the split vectors associated with the nodes in the unique path from $v$ to the root, see \cite{devroye3} and \cite{holmgren2}. This means in particular that $W_{r,v}\stackrel{d}{=}V$.
An application of the Chebyshev inequality gives that $n_v$ for $v$ at depth $d$ is close to
\begin{align}\label{sub} M_v^n:=nW_{1,v}W_{2,v} \dots W_{d,v},
\end{align} 
see \cite{holmgren2}.
Since the $n_v$'s (conditioned on the split vectors) for all $v$ at the same depth are identically distributed, we sometimes skip the vertex index of $W_{r,v}$ in (\ref{binomial}) and just write $W_r$.

\subsubsection{Results Obtained by Using Renewal Theory}\label{renewaltheory}

In \cite{holmgren2} we introduce renewal theory in the context of split trees, and in this study we use this theory frequently for the proof of the Main Theorem, i.e., Theorem \ref{thm} below.

For each vertex $v$, where $W_{r,v}\stackrel{d}{=}V$ are the i.i.d.\ random variables defined in Section \ref{Subtrees},
 let $Y_{k,v}:=-\sum_{r=1}^{k}\ln {W_{r,v}}$. Below we skip the vertex index and just write $Y_k$, 
since for vertices $v$ on the same level $k$ the $Y_{k,v}$'s are identically distributed. 
This is the corresponding notation, as the one we use in \cite{holmgren} for the specific case of the binary search tree, where we define $Y_k:=-\sum_{r=1}^{k}\ln U_r$, where  $U_r$ are uniform $U(0,1)$ random variables.
Recall from (\ref{sub}) in Section \ref{Subtrees}  that the subtree size $n_v$ for a vertex $v$ at depth $k$ is close to $M_v^n$ and note that \begin{align*}M_v^n:=nW_{1,v}W_{2,v} \dots W_{k,v}=ne^{-Y_k}.\end{align*}
Recall that in a binary search tree, the split vector is distributed as $(U,1-U)$ where $U$ is uniform $U(0,1)$ random variable. 
For the binary search tree, the sum $\sum_{r=1}^{k}\ln U_r$ is distributed as a $-\Gamma(k,1)$ random variable. 
For general split trees we do not know the common distribution function of $Y_k$, instead we use renewal theory. (For an introduction to renewal theory, see e.g., \cite[Chapter II]{gut2} or \cite{asmussen}.)
We define the renewal function 
\begin{align}\label{renewal function}
 U(t)=\sum_{k=1}^{\infty}b^k\mathbf{P}(Y_k\leq t)=\sum_{k=1}^{\infty}F_k,
\end{align}
and also denote $F(t):=F_1(t)=b\mathbf{P}(-\ln {W_{r,v}}\leq t)$, which in contrast to standard renewal theory is not a probability measure.
For $U(t)$ we obtain the following renewal equation
\begin{align*}
 U(t)&=F(t) +\sum_{k=1}^{\infty}(F_k\ast F)(t)=F(t)+(U\ast F)(t).
\end{align*}

Recall the definitions of the constants $\mu$ and $\sigma$ in (\ref{splitdef}).
In \cite[Lemma 3.1] {holmgren2} we show the following result which is fundamental for the proof of Theorem \ref{thm}.
Let $t\rightarrow\infty$, then the renewal function $U(t)$
in (\ref{renewal function}) has the solution
\begin{align}\label{renewal equation3}
 U(t)=(\mu^{-1}+o(1))e^t.
\end{align}
In \cite{holmgren2} we also define 
\begin{align*}W(x):=\int_{0}^{x}e^{-t}(U(t)-\mu^{-1}e^t)dt,
\end{align*}
and in \cite[Corollary 3.2]{holmgren2}, we show that
\begin{align} \label{V}W(x)=\frac{\sigma^2-\mu^2}{2\mu^2}-\mu^{-1}+o(1),~~\mathrm{as}~x\rightarrow \infty.
\end{align}

\section{Proofs}\label{Proofs}
\subsection{Notation}\label{notation}
Most of our notation are similar to the ones that we use in \cite{holmgren}, where the binary search tree is considered.

We use the notation $\log_b$ for the $b$-logarithm (recall that a split tree with parameter $b$ is a $b$-ary tree) and $\ln$ for the $e$-logarithm. Let
 $\{x\}=x-\lfloor x \rfloor$ be the fractional part of a real number $x$. 
We treat the case $X_v(T^n)$ in Theorem \ref{thm} in detail and then indicate why the same result holds for $X_e(T^n)$ too.
From now on since it is clear that we consider the vertex model we just write $X(T^n)$. 
First let $X(T^n)_y$ be $X(T^n)-1$ conditioned on the root label $\lambda{ _\sigma }=y$.

We write $\stackrel{d}=$ for equality in distribution.

We say that, $Y_n=o_p(a_n)$ if $a_n$ is a positive number and 
$Y_n$ is a random variable such that $Y_n/a_n \stackrel{p}{\rightarrow}0$ as $n\rightarrow \infty$. 

We say that, $Y_n=\mathcal O _{L^p}(a_n)$ if $a_n$ is a positive number and $Y_n$ is a random variable such that 
 $(\mathbf{E}({Y_n}^p))^{\frac{1}{p}}\leq Ca_n$ for some constant $C$. 
 
We sometimes use the notation $m=\mu ^{-1}\ln{n}$.
 For simplicity in the proofs below we write $\ln{n}$ when we mean $\max \{1,\ln{n}\}$.
 
In the sequel we write $T$ instead of $T^n$. 

For a vertex $v\in T$, we let $T_v$ be the subtree of $T$ rooted at $v$. Recall that $n_v$ is the number of balls and similarly
let $N_v$ be the number of nodes in $T_v$.  

We write Exp$(\theta )$ for an exponential  distribution with parameter $\theta$, i.e., the density function $f(x)=\frac{e^{-\frac{x}{\theta }}}{\theta }$.
We can without loss of generality assume that the labels $\lambda{ _v}$ 
have an exponential  distribution Exp$(1)$. As mentioned above this does not affect the distribution of $X(T^n)$.

Let $d(v)$ denote the depth of $v$, i.e., distance to the root.

Recall that $V$ is a random variable distributed as the identically distributed components in the 
split vector $\mathcal{V}=(V_1,\dots, V_b)$. Also recall that for each vertex $v$ we 
let $Y_{k,v}:=-\sum_{r=1}^{k}\ln {W_{r,v}}$,
 where $W_{r,v}\stackrel{d}{=}V$ are the i.i.d.\ random variables defined in Section \ref{Subtrees}.
  Since the $Y_{k,v}$'s are identically distributed for vertices at the same depth (or depth), 
we sometimes skip the vertex index and just write $Y_k$. 
Recall from (\ref{renewal function}) that we define the renewal function 
 $U(t):=\sum_{k=1}^{\infty}b^k\mathbf{P}(Y_k\leq t)$.

 Let $\Lambda_{v_i}$ be the minimum of $\lambda _v$ along the path 
$P(v_i)=\sigma,\ldots,v_i$, from 
the root $\sigma $ of $T$ to $v_i,$ $1\leq i \leq b^L$, where $v_i$ are the vertices at 
depth $L=\lfloor \beta \log_b\ln{n}\rfloor$ for some constant $\beta$.  
Thus, the definition of $\Lambda_{v_i}$ and the assumption $\lambda _v\stackrel{d}=$  
Exp$(1)$ give $\Lambda_{v_i}\stackrel{d}=$  Exp$(\frac{1}{L+1})$.  

For simplicity we write $T_i:=T_{v_i},$ $n_i:=n_{v_i}$, $N_i:=N_{v_i}$ and $\Lambda_i:=\Lambda_{v_i}$. We also let  ${T_i}_v$ denote a subtree of $T_i$ rooted at $v$ (note that ${T_i}_v$ is $T_v$ for $v\in T_i$).  Let ${n_i}_v$ denote the number of balls in ${T_i}_v$. 

We write $d_i(v):=d(v)-L$ (i.e., the depth in the subtree $T_i,~i\in\{1,\ldots,b^L\}$, of a vertex $v\in T_i$).
  
 We say that a vertex $v$ in $T^n$ is ''good''  if 
\begin{align*}\mu^{-1}\ln{n}-\ln^{0.6}{n}\leq d(v)\leq \mu^{-1}\ln{n}+\ln^{0.6}{n},
\end{align*}
and  otherwise it is bad.
  In particular a vertex $v\in T_i$ is ''good''  if 
\begin{align}\label{good}\mu^{-1}\ln{n_i}-\ln^{0.6}{n_i}\leq d_i(v)\leq \mu^{-1}\ln{n_i}+\ln^{0.6}{n_i},
\end{align}
and  otherwise it is bad.

We define $\varphi(T_i,\Lambda_i):=\mathbf{E}(X(T_i)_{\Lambda_i}\mid T_i,\Lambda_i)$ (the conditional expected value of $X(T_i)_{\Lambda_i}$ 
given the tree $T_i$ and $\Lambda_i$). (We can think of $X(T_i)_{\Lambda_i}$ as $X(T_{i})-1$ conditioned on the root 
label $\lambda _{v_i }=\Lambda_i$.) Similarly we let $\psi(T_i,\Lambda_i):=\mathbf{Var}(X(T_i)_{\Lambda_i}\mid T_i,\Lambda_i)$ for vertices  $v_i,~1\leq i \leq b^L$ (the conditional variance of $X(T_i)_{\Lambda_i}$ given the tree $T_i$ and $\Lambda_i$). 

The conditional expected value of a random variable $Z$ given the subtree size $n_i$ of $T_i$ is denoted by
$\mathbf{E}_{n_i}(Z):=\mathbf{E}(Z\mid n_i)$.

We write  $\xi_v:=\frac{n_v\mu^{-1}\ln{n}}{n}\cdot e^{-\lambda_v\mu^{-1}\ln{n}}$, which is used in the later part of the proof when we consider triangular arrays.

We use the notation $\Omega_L$ for the $\sigma$-field generated by $\{n_v,~d(v)\leq L\}$.
Finally, we write $\mathscr{G}_j$  
as the $\sigma$-field generated by the $\mathcal{V}$ vectors for all vertices $v$ with $d(v)\leq j$. Equivalently, this is the $\sigma$-field generated by $\{W_{r,v},~r\in\{1,2,\dots,j\}\}$, for all vertices $v$ with $d(v)= j$.
 In particular we use that the subtree sizes $\{n_v,~d(v)\leq L\}$ up to small errors are determined by the $\sigma$-field $\mathscr{G}_L$; this follows because of the representation of subtree sizes in Section \ref{Subtrees}.

\subsection{Expressing the normalized number of records as a sum of triangular arrays}

Recall from Section \ref{notation} that we define $\varphi(T_i,\Lambda_i):=
\mathbf{E}(X(T_i)_{\Lambda_i}\mid T_i,\Lambda_i)$, where $T_i$ is 
the subtree rooted at $v_i$ at depth $L$ and $\Lambda_i$ is the minimum of 
$\lambda _v$ in the path from $v_i$ to the root $\sigma$ of $T$. 

\begin{Lemma}\label{lem1} For all subtrees $T_i$ rooted at $v_i$ with $d(v_i)=L$, conditioned on the subtree
size $n_i$,
\begin{multline}\label{eq2} \varphi(T_{i},\Lambda_i)=\frac {N_i}{\mu^{-1}\ln{n_i}}(1-e^{-(\mu^{-1}\ln{n_i})\Lambda_i})-\frac {\Upsilon(T_i)-\mu^{-1}N_i\ln{n_i}}{\mu^{-2}\ln^2{n_i}} +\\
\sum_{good~v\in T_i}\frac {{(d_i(v)-\mu^{-1}\ln{n_i})^2}}{ \mu^{-3}\ln^3{n_i}} +
\mathcal O _{L^1}\Big(\frac {n_i}{{\ln^{2.2}{n_i}}}\Big),
\end{multline} 
 where $\Upsilon(T_i)$ is the total path length  of the tree $T_i$, and the good vertices $v\in T_i$ are those with $d_i(v)$ satisfying (\ref{good}).
\end{Lemma} 

\begin{proof}

Let for each vertex $v\in T_i$, $I_v$ 
 be the indicator that $\lambda{ _v}$  is the minimum value given $T_i$ and $\Lambda_i$.
We get $\varphi(T_{i},\Lambda_i)=\sum_{v\neq v_i }\mathbf{E}(I_v)$. 
If $d_i(v)=j$ in $T_i$, let ${v_i,v_{i1},...,v_{ij}=v}$ be the vertices in the path from the root 
$v_i$ to $v$. Then, $I_v =1$,
if and only if, $\lambda{ _{v_{ij}}}<\Lambda_i$ and $\lambda{ _{v_{ik}}}>\lambda{ _{v_{ij}}}$ for $k\in {\{1,\dots,j-1\}}$.
 Since the $\lambda{ _{v} }$'s (for all vertices $v$ in $T_i$) are independent Exp$(1)$ random variables
 \begin{equation}\label{indicator}\mathbf{E}(I_v)=\int_{0}^{\Lambda_i}\prod_{k=1}^{j-1}\mathbf{P}(\lambda{ _{v_ik}}>x)
e^{-x}dx=\int_{0}^{\Lambda_i}e^{-jx}dx=\frac {1-e^{-j\Lambda_i}}{j}.\end{equation} Thus,
\begin{equation*}\varphi(T_i,\Lambda_i)=\sum_{v\neq v_i}\frac {1-e^{-d_i(v)\Lambda_i}}{d_i(v)}.
\end{equation*}

Expanding $\frac {1}{d_i(v)}$ for arbitrary good $v\in T_i$ gives
\begin{align*}&
\frac {1}{d_i(v)}=\frac {1}{\mu^{-1}\ln{n_i}}-
\frac {d_i(v)-\mu^{-1}\ln{n_i}}{\mu^{-2}\ln^2{n_i}}+\frac {{(d_i(v)-\mu^{-1}\ln{n_i})}^2}{\mu^{-3}\ln^3{n_i}}
\nonumber\\&+\mathcal O\Big( \frac {\mid (d_i(v)-\mu^{-1}\ln{n_i})^3\mid }{{\ln^4{n_i}}} \Big).\end{align*}
 
Recall from Section \ref{nodes}, that the number of bad vertices in $T_i$, i.e., those that are 
not in the strip in (\ref{good}), is $\mathcal O_{L^{1}} \Big(\frac {n_i}{\ln^{3}{n_i}}\Big)$ 
and can thus be ignored.
 Thus, summing over all nodes $v\in T_i$ gives \begin{multline}\label{eq:height3}
\sum_{v\neq v_i }\frac {1}{d_i(v)} =\sum_{good~v\neq v_i }\frac {1}{d_i(v)}+\mathcal O_{L^{1}} \Big(\frac {n_i}{\ln^{3}{n_i}}\Big)= \frac {N_i}{\mu^{-1}\ln{n_i}}-\\
\frac {\Upsilon(T_i)-\mu^{-1}N_i\ln{n_i}}{\mu^{-2}\ln^2{n_i}}+\sum_{good~v \in T_i}\frac {{(d_i(v)-\mu^{-1}\ln{n_i})^2}}{ \mu^{-3}\ln^3{n_i}} 
+\mathcal O_{L^{1}} \Big(\frac {n_i}{\ln^{2.2}{n_i}}\Big).\end{multline}


Now we prove that
\begin{equation}\label{eq:3}\sum_{v\neq v_i }\frac {1}{d_i(v)}(e^{-d_i(v)}-e^{-(\mu^{-1}\ln{n_i})\Lambda_i})
=\mathcal O _{L^1}\Big(\frac {n_i}{\ln^{2.2}{n_i}}\Big),
\end{equation} which obviously implies, \begin{equation}\label{eq:3,1}\varphi(T_i,\Lambda_i)=(1-e^{-(\mu^{-1}\ln{n_i})\Lambda_i})\sum_{v\neq v_i }\frac {1}{d_i(v)}+ 
\mathcal O _{L^1}\Big(\frac {n_i}{\ln^{2.2}{n_i}}\Big).
\end{equation}

For simpler calculations we show the bound in (\ref{eq:3}) by considering,
$e^{-\Lambda_i\lfloor \mu^{-1}\ln{n_i}\rfloor}$ instead of $e^{-(\mu^{-1}\ln{n_i})\Lambda_i}$.
That one can do this is because multiplying the Taylor estimate in (\ref{eq:height3}) 
by $e^{-\Lambda_i\lfloor \mu^{-1}\ln{n_i}\rfloor}$, gives the same expression up to the error term 
$\mathcal O _{L^1}\Big(\frac {n_i}{\ln^{2.2}{n_i}}\Big)$ as multiplying by $e^{-(\mu^{-1}\ln{n_i})\Lambda_i}$. For $j>0$, \begin{displaymath}
e^{(-\lfloor\mu^{-1}\ln{n_i}\rfloor +j)\Lambda_i}=e^{-\Lambda_i\lfloor\mu^{-1}\ln{n_i}\rfloor}+e^{(-\lfloor\mu^{-1}\ln{n_i}\rfloor +j)\Lambda_i}
(1-e^{-j\Lambda_i})
\end{displaymath} and \begin{align*}
&e^{(-\lfloor\mu^{-1}\ln{n_i}\rfloor -j)\Lambda_i}=e^{-\Lambda_i\lfloor \mu^{-1}\ln{n_i}\rfloor}+e^{(-\lfloor\mu^{-1}\ln{n_i}\rfloor +j)\Lambda_i}
(e^{-2j\Lambda_i}-e^{-j\Lambda_i}).
\end{align*}

 Since we only have to consider the good vertices it is enough to show that 
\begin{align}\label{profile2}
Q_1+Q_2=\mathcal O _{L^1}\Big(\frac {n_i}{\ln^{2.2}{n_i}}\Big),
\end{align} 
where 
\begin{align*}
 Q_1:=\sum_{j=1}^{\lfloor \ln^{0.6}{n_i}\rfloor }\sum_{d_i(v)=j}e^{(-\lfloor\mu^{-1}\ln{n_i}\rfloor +j)\Lambda_i}\cdot (1-e^{-j\Lambda_i})
\cdot \frac {1}{\lfloor \mu^{-1}\ln{n_i}\rfloor -j},
\end{align*}
\begin{align*}
 Q_2:=\sum_{j=1}^{\lfloor\ln^{0.6}{n_i}\rfloor}\sum_{d_i(v)=j} \nonumber
e^{(-\lfloor\mu^{-1}\ln{n_i}\rfloor +j)\Lambda_i}\cdot 
(e^{-2j\Lambda_i}-e^{-j\Lambda_i})\cdot \frac {1}{\lfloor \mu^{-1}\ln{n_i}\rfloor +j}.
\end{align*}

We have
\begin{align}\label{compare1}Q_1&
\leq N_i e^{(-\lfloor\mu^{-1}\ln{n_i}\rfloor +\ln^{0.6}{n_i})\Lambda_i}\cdot (1-e^{-\ln^{0.6}{n_i}\Lambda_i})
\cdot \frac {1}{\lfloor \mu^{-1}\ln{n_i}\rfloor -\ln^{0.6}{n_i}}\nonumber\\&=
N_i\mathcal O\Big(\frac {\ln^{0.6}{n_i}\Lambda_i}{ \mu^{-1}\ln{n_i}}\Big)  
e^{(-\lfloor \mu^{-1}\ln{n_i}\rfloor +\ln^{0.6}{n_i})\Lambda_i}.
\end{align}
and similarly
\begin{align*}&Q_2 = N_i\mathcal O\Big(\frac {\ln^{0.6}{n_i}\Lambda_i}{ \mu^{-1}\ln{n_i}}\Big)  
e^{(-\lfloor \mu^{-1}\ln{n_i}\rfloor +\ln^{0.6}{n_i})\Lambda_i}.
\end{align*}

Since $\Lambda_i$ is Exp$(\frac{1}{L+1})$ random variable, we get that
\begin{align*} &\mathbf{E}\big(\Lambda_ie^{(-\lfloor \mu^{-1}\ln{n_i}\rfloor +\ln^{0.6}{n_i})\Lambda_i}\big)
=\int_{0}^{\infty}
(L+1)ye^{(-\lfloor \mu^{-1}\ln{n_i}\rfloor +\ln^{0.6}{n_i})y}e^{-y(L+1)}dy=
\nonumber\\&\Big|\frac {(L+1)ye^{(-\lfloor\mu^{-1}\ln{n_i}\rfloor +\ln^{0.6}{n_i}-(L+1))y}}{-\lfloor\mu^{-1}\ln{n_i}\rfloor 
+\ln^{0.6}{n_i}-L-1}\Big|_{0}^{\infty}-
\int_{0}^{\infty}\frac {(L+1)e^{(-\lfloor \mu^{-1}\ln{n_i}\rfloor +\ln^{0.6}{n_i}-(L+1))y}}{-\lfloor \mu^{-1}\ln{n_i}\rfloor +\ln^{0.6}{n_i}
-L-1}dy \nonumber\\&=\frac {L+1}{(\lfloor\mu^{-1}\ln{n_i}\rfloor -\ln^{0.6}{n_i}+L+1)^2},
\end{align*} 
 Thus, (\ref{profile2}) holds and it follows that $(\ref{eq:3,1})$ is satisfied. 

Now we show that $(\ref{eq:3,1})$ implies (\ref{eq2}) in Lemma \ref{lem1}. We have $e^{-(\mu^{-1}\ln{n_i})\Lambda_i}=\mathcal O _{L^1} (\frac {L}{\ln{n_i}})$. Hence, 
\begin{align*}e^{-(\mu^{-1}\ln{n_i})\Lambda_i} \cdot \sum_{good~v\in T_i}\frac {{(d_i(v)-\mu^{-1}\ln{n_i})^2}}{ \mu^{-3}\ln^3{n_i}}=\mathcal O _{L^1}\Big(\frac {n_i}{\ln{n_i}}\Big).
 \end{align*}
 Recall from Section \ref{nodes} that the number of bad nodes in $T_i$ is $\mathcal O _{L^1}(\frac {n_i}{\ln^{3}{n_i}})$ and that for any constant $r$ there is a constant $C>0$ such that the number of nodes with $d(v)\geq C \ln n$  
is $\mathcal{O}_{L^{1}}\Big(\frac{1}{n^r}\Big)$. By using these facts we get an obvious 
upper bound of the total path length, i.e.,
\[\mid \Upsilon(T_i)-\mu^{-1} N_i\ln{n_i}\mid \leq N_i\ln^{0.6}{n_i}+\mathcal O _{L^1}\Big(\frac {n_i}{\ln{n_i}}\Big).\]
Hence, \[ \frac {(\Upsilon(T_i)-\mu^{-1}N_i\ln{n_i})e^{-(\mu^{-1}\ln{n_i})\Lambda_i}}{\mu^{-2}\ln^2{n_i}}
=\mathcal O _{L^1}\Big(\frac {n_i}{\ln^{2.2}{n_i}}\Big), \] 
and Lemma \ref{lem1} follows. 
\end{proof}

Recall from Section \ref{notation} that we define $\psi(T_i,\Lambda_i):=\mathbf{Var}(X(T_i)_{\Lambda_i}\mid T_i,\Lambda_i)$, 
and that we write 
 $\mathbf{E}_{n_i}(.):=\mathbf{E}(.|n_i)$ for the conditional expected value given $n_i$.

\begin{Lemma}\label{lem2} For all vertices $v_i$ with $d(v_i)=L$, conditioned on $n_i$,
 \begin{equation*}\mathbf{E}_{n_i}(\psi(T_i,\Lambda_i))=\mathcal O \Big(\frac {n_{i}^2}{{\ln^3{n_i}}}\Big).\end{equation*} 
\end{Lemma}

\begin{proof}
For all vertices $v\in T_i$, let $I_{v}$ 
be the same indicator as in the proof of Lemma \ref{lem1} above. 
Suppose that $v$ and $w$ are two vertices in $T_i$ at depth $d_i(v)=j,~d_i(w)=k$ 
with last common ancestor at depth $d_i(u)=d$. Suppose first that $d<j,d<k$.  Let $\{v_i,u_1, \ldots,u_d=u\}$ 
be the vertices in the path from $v_i$ to $u$ and let $Z=\min\{\lambda _{u_s}:1\leq s\leq d\}$. 
Conditioned on $Z$, $I_v$ and $I_w$ are independent.  Let $Z\wedge \Lambda_i$ denote the minimum of $Z$ and $\Lambda_i$.
Since $v$ has depth $j-d$ above $u$, (\ref{indicator}) yields \[\mathbf{E}(I_v\mid Z)=\frac {1-e^{-(j-d)
(Z\wedge \Lambda_i)}}{j-d},\]
and similarly for  $I_w$. (Compare this with \cite[Lemma 2.2]{holmgren}.) As in \cite[equation (18)]{holmgren}, 
\begin{eqnarray}\label{indicator2}\mathbf{E}(I_vI_w)=\frac {1}{j-d}\frac {1}{k-d}\Bigg(1-e^{-d\Lambda_i}-
\frac {d}{j}(1-e^{-j\Lambda_i})-\frac {d}{k}(1-e^{-k\Lambda_i})+\nonumber\\\frac {d}{j+k-d}(1-e^{-(j+k-d)\Lambda_i})+
e^{-d\Lambda_i}-e^{-j\Lambda_i}-e^{-k\Lambda_i}
+e^{-(j+k-d)\Lambda_i}\Bigg).\end{eqnarray} 

The covariance of $I_v$ and $I_w$ is 
\[\mathbf{Cov}(I_v,I_w)=\mathbf{E}(I_vI_w)-\mathbf{E}(I_v)\mathbf{E}(I_w).\] We say that a pair $(v,w)$ is "good" if $j$ and $k$ satisfy 
\begin{align*} \mu ^{-1}\ln{n_i}-\ln^{0.6}{n_i}\leq j,k \leq \mu ^{-1}\ln{n_i}+\ln^{0.6}{n_i},\end{align*} 
 and otherwise it is "bad". 
 From \cite[equation (7)]{jan1} by (\ref{indicator}) and (\ref{indicator2}) above, for a good pair
 \begin{align}\label{indicator3}\mathbf{Cov}(I_v,I_w) &=
\frac {1}{jk}e^{-(j+k-d)\Lambda_i}(1-e^{-d\Lambda_i})+\mathcal O \Big(\frac {d}{{\ln^3{n_i}}}\Big)=\mathcal O _{L^1}\Big(\frac {d}{{\ln^3{n_i}}}\Big).
\end{align} 
(Compare this with \cite[equation (19)]{holmgren}.)  Since the number of bad vertices is $\mathcal O_{L^1} (\frac{n_i}{\ln^3{n_i}})$ it follows that the number of bad pairs, is $\mathcal O_{L^1} (\frac{n_{i}^{2}}{\ln^3{n_i}})$.
Hence, because of the obvious upper bound that $\mathbf{Cov}(I_v,I_w)$ is at most 1, the sum of covariances for the bad pairs is 
 $\mathcal O \Big(\frac {n_i^2}{{\ln^3{n_i}}}\Big)$.
 Thus,
 \begin{align}\label{var3}\mathbf{E}_{n_i}(\psi(T_i,\Lambda_i))=\mathbf{E}_{n_i}\Big(\sum_{good~~(v,w)\in T_i}\mathbf{Cov}(I_v,I_w)\Big)+\mathcal O \Big(\frac {n_i^2}{{\ln^3{n_i}}}\Big).
 \end{align}

Recall that $\mathscr{G}_j$ is the $\sigma$-field generated by the split vectors for all vertices $v$ with $d(v)\leq j$.
 Recall the representation of subtree sizes in split trees described in (\ref{binomial}) in Section \ref{Subtrees}.
Recall that ${n_i}_{v}$ denotes the number of balls in the subtrees rooted at $v$ for $v\in T_i$. From (\ref{binomial})  we get that for $v$, where $d_i(v)=d$, 
\begin{align*}
 \mathbf{E}_{n_i}({n_{i}}_v|\mathscr{G}_{L+d})\leq n_i\prod_{r=1}^dW_r+s_1d.
\end{align*}
 Thus, 
\begin{equation*}\mathbf{E}_{n_i}({n_{i}}_v)\leq n_i\prod_{r=1}^d {\mathbf{E}({W_r})}
+ds_1=\frac{n_i}{b^d}+ds_1.\end{equation*} 

Again by using (\ref{binomial}) we get that 
\begin{align*}
 \mathbf{E}_{n_i}({n_{i}}_v^2|\mathscr{G}_{L+d})=n_i^{2}\prod_{r=1}^dW_r^{2}+\mathcal{O}(n_id\prod_{r=1}^dW_r)+\mathcal{O}(d^2).
\end{align*}
 Thus, 
\begin{equation}\label{subtree2,2}\mathbf{E}_{n_i}({n_{i}}_v^2)\leq n_i^2\prod_{r=1}^d {\mathbf{E}({W_r}^2})+\mathcal{O}\Big(\frac{n_id}{b^d}\Big)+\mathcal{O}(d^2).\end{equation}

Note that $\mathbf{E}(W_r^2)<\mathbf{E}(W_r)=\frac{1}{b}$ since $W_r\in [0,1]$.  Hence, there is an 
$\epsilon>0$ such that the right hand-side in (\ref{subtree2,2}) is bounded by
\begin{equation}\label{subtree2,3}\frac{n_i^2}{(b+\epsilon)^d}+\mathcal{O}\Big(\frac{n_id}{b^d}\Big)+\mathcal{O}(d^2).\end{equation}

From (\ref{var3}) by using (\ref{indicator3}),
 (\ref{subtree2,2}) and (\ref{subtree2,3}),
\begin{align*}\mathbf{E}_{n_i}(\psi(T_i,\Lambda_i))&=
\mathcal O \Big(\sum_{d}
\frac {n_i^2\cdot b^d\cdot d}{(b+\epsilon )^d{\ln^3{n_i}}}\Big)+\mathcal O \Big(\frac {n_i^2}{{\ln^3{n_i}}}\Big)
=\mathcal O \Big(\frac {n_i^2}{{\ln^3{n_i}}}\Big).\end{align*} \end{proof}

 The estimate in Lemma \ref{lem2} is used in the proof of the following result.

\begin{Lemma}\label{lem3} In a split tree $T^n$, let ${v_i},1\leq i\leq b^L$,
 be the vertices at depth $L=\lfloor \beta \log_b\ln{n}\rfloor$ choosing $\beta>\frac{1}{-\log_b\mathbf{E}(V^2)-1}$. 
 Then 
\begin{equation*}X(T^n)=\sum_{i=1}^{b^L}\varphi(T_i,\Lambda_i)
+o_p\Big(\frac {n}{{\ln^2{n}}}\Big). 
\end{equation*} 
\end{Lemma}
 
\begin{proof}
We write the number of records as $\{P^{\ast }+P_1+\ldots+P_{b^L} \}$, where $P^{\ast }$ is 
the number of records with depth at most $L$ and $P_i$ is the number of records in the subtree $T_i$ rooted at depth $L$, except for the root $v_i$.
Let $\mathscr F_{L}$ be the $\sigma$-field generated 
by $\{\lambda _v:d(v)\leq L\}$ and $\mathscr F_{L}^{\ast}$ the $\sigma$-field generated by $T^n$ and $\mathscr F_{L}$. We also note that $\mathbf{E}(P_i\mid \mathscr F_{L}^{\ast}) =
\varphi(T_i,\Lambda_i)$. 
By the same calculation as in \cite[equation (22)]{holmgren},
\begin{align}\label{sum2} 
&\mathbf{E}\bigg{(}\Big(X(T^n)-P^{\ast }-
\sum_{i=1}^{b^L}{\varphi(T_i,\Lambda_i)}{\Big)}^2 \bigg| \mathscr F_{L}^{\ast} \bigg{)}
=\sum_{i=1}^{b^L}\psi(T_i,\Lambda_i).
\end{align} 

Taking the expectation of the conditional expected value in (\ref{sum2}) yields 
\begin{align}\label{sum3} 
&\mathbf{E}\bigg{(}\Big(X(T^n)-P^{\ast }-
\sum_{i=1}^{b^L}{\varphi(T_i,\Lambda_i)}{\Big)}^2 \bigg{)}
=\sum_{i=1}^{b^L}\mathbf{E}\psi(T_i,\Lambda_i).
\end{align}
 
We observe the obvious fact that the sum of those $n_i,~i\in\{1,\dots, b^L\}$,
that are less than $\frac{n}{b^{kL}}$ for $k$ large enough, is bounded by
\begin{align}\label{bound}b^L \cdot\frac{n}{b^{kL}}=\mathcal{O}\Big(\frac{n}{\ln^3 n}\Big).
\end{align}
(Note that by choosing $k$ large enough in (\ref{bound}) 
the power of the logarithm can be taken arbitrarily large.)
Lemma \ref{lem2} and (\ref{bound}) 
give that \begin{align}\label{sum3,1}\sum_{i=1}^{b^L}\mathbf{E}_{n_i}(\psi(T_i,\Lambda_i))=\mathcal O \Big(\sum_{i=1}^{b^L}{\frac {n_i^2}{{\ln^3{n}}}}\Big).
\end{align}
(Compare this with \cite[equation (25)]{holmgren}.)
The expected value  of the sum in (\ref{sum3,1}) is equal to the expected value of the left hand-side in 
(\ref{sum3}).
From the calculations in (\ref{subtree2,2})  above for $i \in \{1,\ldots, b^L\}$,
\begin{equation}\label{subtree2,4}\mathbf{E}(n_i^2)\leq n^2(\mathbf{E}(V^2))^L+\mathcal{O}(nL).\end{equation}

Hence, choosing $\beta>\frac{1}{-\log_b\mathbf{E}(V^2)-1}$ one gets from (\ref{subtree2,4}) that
\begin{align}\label{sum4}&\sum_{i=1}^{b^L}\mathbf{E}(n_i^2)=o(\frac{n^2}{\ln n}), \end{align}
  and thus the left hand-side in (\ref{sum3}) is $o(\frac{n^2}{\ln^4 n})$.
Thus, Lemma \ref{lem3} follows from the well-known Markov inequality.
\end{proof}
Applying Lemma \ref{lem1} and Lemma \ref{lem3} yields for $\beta>\frac{1}{-\log_b\mathbf{E}(V^2)-1}$ that
 \begin{multline}\label{nicesum1,1}X(T^n)=\sum_{i=1}^{b^L}\bigg(\frac{2N_i}{\mu^{-1}\ln{n_i}}-\frac {N_i e^{-(\mu^{-1}\ln{n_i})\Lambda_i}}{\mu^{-1}\ln{n_i}}
 +\sum_{good~v \in T_i}\frac {{(d_i(v)-\mu^{-1}\ln{n_i})^2}}{ \mu^{-3}\ln^3{n_i}}
 \\-\frac {\Upsilon (T_{i})}{\mu^{-2}{\ln^2{n_i}}}\bigg) +
o_p\Big(\frac {n}{{\ln^2{n}}}\Big),
\end{multline}
where we used that the Markov inequality gives $\mathcal O _{L^1}\Big(\frac{n}{\ln^{2.2}{n}}\Big)=o_p\Big(\frac {n}{{\ln^2{n}}}\Big)$.

In \cite[Corollary 2.2]{holmgren2} we prove that  
\begin{align}\label{indsum3}&\sum_{i=1}^{b^L}\sum_{good~v\in T_i}\frac {{(d_i(v)-\mu^{-1}\ln{n_i})^2}}{ \mu^{-3}\ln^3{n_i}}=\frac{\sigma^2\alpha n}{\ln^2{n}}+o_p\Big(\frac {n}{{\ln^2{n}}}\Big).\end{align}

We get for  $n_i \geq \frac{n}{b^{kL}}$,
\begin{align*}&\mathbf{E}_{n_i}\Big(\Big| e^{-(\mu^{-1}\ln{n_i})\Lambda_i}-e^{-(\mu^{-1}\ln{n})\Lambda_i}\Big|\Big)
\\&=\frac{L+1}{L+1+\mu^{-1}\ln{n_i}}-\frac{L+1}{L+1+\mu^{-1}\ln{n}}\\&=\mathcal O \Big(\frac{L^2}{\ln^2{n}}\Big),
\end{align*}
and it follows that
\begin{align}\label{pontus2}\mathbf{E}\Big(\Big| \frac {N_i}{\mu^{-1}\ln{n_i}}e^{-(\mu^{-1}\ln{n_i})\Lambda_i}-\frac {N_i}{\mu^{-1}\ln{n}}e^{-(\mu^{-1}\ln{n})\Lambda_i}\Big|\Big)
=\mathcal O \Big(\frac{L^2n}{b^{L}\ln^3{n}}\Big).\end{align}

Again we use the bound in (\ref{bound}) for those $n_i <\frac{n}{b^{kL}}$ (for large enough $k$) 
so that we can ignore them in the sums in (\ref{nicesum1,1}).
Thus, by (\ref{indsum3}) and (\ref{pontus2}) with another application of the Markov inequality,
 the approximation in (\ref{nicesum1,1}) can be simplified to 
\begin{multline}\label{goodsumhej} X(T^n)= \sum_{i=1}^{b^L}\frac {2N_i}{\mu^{-1}\ln{n_i}}-
 \sum_{i=1}^{b^L}\frac {\Upsilon(T_i)}{\mu^{-2}{\ln^2{n_i}}}\\
 -\frac {1}{\mu^{-1}\ln{n}} \sum_{i=1}^{b^L}N_i e^{-(\mu^{-1}\ln{n})\Lambda_i}
 +\frac {\sigma^2\alpha n}{\ln^2{n}}+
o_p\Big(\frac {n}{{\ln^2{n}}}\Big).
\end{multline} 
(Compare this with  \cite[equation(27)]{holmgren}.)

By choosing $\beta$ large enough we can sharpen the error term in (\ref{sum4}), i.e.,
\begin{align}\label{error}
\sum_{i=1}^{b^L}\mathbf{E}\Big(n_i^{2}\Big)=o\Big(\frac{n^2}{\ln ^k n}\Big),
\end{align}
for arbitrary large $k$.
Applying (\ref{error}), the variance result in (\ref{vertexball2}), and assuming (A3), Chebyshev's inequality results in
\begin{align}\label{error2}
\sum_{i=1}^{b^L}\frac{N_i}{\ln n_i}=
\sum_{i=1}^{b^L}\frac{\alpha n_i}{\ln n_i}+o_p\Big(\frac {n}{{\ln^2{n}}}\Big).
\end{align}
The third sum in  (\ref{goodsumhej}) is treated similarly.
For simplicity (in the calculations below) we change the notation 
$N_i,~1\leq i \leq b^L$, to $N_v,~d(v)=L$, and similarly for $n_i,~1\leq i \leq b^L$.
Hence, from (\ref{goodsumhej}), for $\beta$ large enough, we get
\begin{multline}\label{goodsum} X(T^n)= \sum_{i=1}^{b^L}\frac {2\alpha n_i}{\mu^{-1}\ln{n_i}}-
 \sum_{i=1}^{b^L}\frac {\Upsilon(T_i)}{\mu^{-2}{\ln^2{n_i}}}\\
 -\frac {1}{\mu^{-1}\ln{n}} \sum_{i=1}^{b^L}\alpha n_i e^{-(\mu^{-1}\ln{n})\Lambda_i}
 +\frac {\sigma^2\alpha n}{\ln^2{n}}+
o_p\Big(\frac {n}{{\ln^2{n}}}\Big).
\end{multline} 

\begin{Lemma}\label{lem4} Let  $L=\lfloor \beta \log_b\ln{n}\rfloor$ for some constant $\beta$,
\[\sum_{i=1}^{b^{L}} n_i e^{-(\mu^{-1}\ln{n})\Lambda_i}=\sum_{d(v) \leq L}n_v e^{-(\mu^{-1}\ln{n})\lambda _{v}}+o_p\Big(\frac {n}{{\ln{n}}}\Big).\]
 Thus, choosing $\beta>\frac{1}{-\log_b\mathbf{E}({V}^2)-1}$ from (\ref{goodsum}),
 \begin{multline*}
X(T^n)=\sum_{i=1}^{b^L}\frac {2\alpha n_i}{\mu^{-1}\ln{n_i}}-
 \sum_{i=1}^{b^L}\frac {\Upsilon(T_i)}{\mu^{-2}{\ln^2{n_i}}}\\
-\frac {1}{\mu^{-1}\ln{n}} \sum_{d(v)\leq L}\alpha n_v e^{-(\mu^{-1}\ln{n})\lambda _{v}}
+\frac {\sigma^2 \alpha n}{\ln^2{n}}+
o_p\Big(\frac {n}{{\ln^2{n}}}\Big).
\end{multline*}  
\end{Lemma}

\begin{proof} Recall that we write $m:=\mu^{-1}\ln{n}$, and  $\Lambda_i$ for the minimum of the $L+1$ i.i.d.\ random variables $\lambda_v,~v\in P(v_i)=\{\sigma,\dots,v_i\}$, where $P(v_i)$ is the path from the root $\sigma$ to $v_i$. Thus, $e^{-m\Lambda_i}$ is the maximum. 
Now we define $\Lambda_i^{j}$ as the $j$-th smallest value in $\{\lambda_v,~v\in P(v_i)\}$, so that  $e^{-m\Lambda_i^{j}}$ is the $j$-th maximum.
Note in particular that $\Lambda_i^{1}=\Lambda_i$.
Choosing $a=\frac{2\ln{m}}{m}$  gives that for some $i$, the probability that 
at least $\lfloor\beta\rfloor+1$ of the $\lambda _{v}$'s, $v~\in~P(v_i)$,  are less than $a$ is \begin{equation*}\mathcal O \Big(b^L L^{\lfloor\beta\rfloor+1} a^{\lfloor\beta\rfloor+1}\Big)= \mathcal O \Big(\frac{b^L\ln^{2(\lfloor\beta\rfloor+1)} m}{m^{\lfloor\beta\rfloor+1}}\Big)=o(1).\end{equation*} Thus, with probability tending to 1, there are at most $\lfloor\beta\rfloor$ values $\lambda _{v}$  less than $a$ in each $P(v_i)$, giving for each $i$,
\begin{equation*} 0\leq \sum_{v\in P(v_i)}e^{-m\lambda _{v}}-\sum_{j=1}^{\lfloor\beta\rfloor}e^{-m\Lambda_i^{j}}\leq (L-\lfloor\beta\rfloor)e^{-ma}=\frac{L-\lfloor\beta\rfloor}{m^2}.
\end{equation*} 
Hence, using that $n_v-sb^L\leq\sum_{i:v\in P(v_i)}n_i\leq n_v$,
\begin{align*} \sum_{i=1}^{b^L}n_i\sum_{j=1}^{\lfloor\beta\rfloor}e^{-m\Lambda_i^{j}}&= \sum_{i=1}^{b^L}n_i\sum_{v\in P(v_i)}e^{-m\lambda _{v}}+o_p\Big(\frac {n}{{\ln{n}}}\Big)
\nonumber\\&=\sum_{d(v)\leq L}e^{-m\lambda _{v}}\sum_{i:v\in P(v_i)}n_i+o_p\Big(\frac {n}{{\ln{n}}}\Big)
\nonumber\\&=\sum_{d(v)\leq L}n_v e^{-m\lambda _{v}}+
o_p\Big(\frac {n}{{\ln{n}}}\Big).
\end{align*}
Observing that the second smallest value $\Lambda_{i}^{2}$ in $i:v\in P(v_i)$, is at most $x$ 
if at least two $\lambda_v$ are at most $x$, and using that the $\lambda_v$'s are i.i.d.\, we calculate the distribution function of  $\Lambda_{i}^{2}$ as
\begin{align*}
\mathbf{P}\Big(\Lambda_{i}^{2}\leq x\Big)&=1-\mathbf{P}(\lambda_v>x)^L-L\mathbf{P}(\lambda_v>x)^{L-1}\mathbf{P}(\lambda_v\leq x)
\nonumber\\&=1-e^{-Lx}-Le^{-(L-1)x}(1-e^{-x}).
\end{align*}
Hence,
\begin{align*}
\mathbf{E}\Big(e^{-m\Lambda_i^{2}}\Big)&=\int_{0}^{\infty}e^{-mx}\Big((L-L^2)e^{-Lx}+L(L-1)e^{-(L-1)x}\Big)dx
\nonumber\\&=\frac{L-L^2}{m+L}+\frac{L^2-L}{m+L-1}=\mathcal{O}\Big(\frac{L^2}{m^2}\Big),
\end{align*}
implying
\begin{align*}\sum_{i=1}^{b^L}n_i\sum_{j=2}^{\lfloor\beta\rfloor} e^{-m \Lambda_i^{j}}=\mathcal{O}_{L^1}(\frac{nL^2}{m^2}).
\end{align*}
Thus, the Markov inequality gives
\begin{align*}\sum_{i=1}^{b^{L}} n_i e^{-(\mu^{-1}\ln{n})\Lambda_i}=\sum_{d(v) \leq L}n_v e^{-(\mu^{-1}\ln{n})\lambda _{v}}+o_p\Big(\frac {n}{{\ln{n}}}\Big).
\end{align*}

\end{proof} 

 Thus, from Lemma \ref{lem4} (where $\beta$ is chosen large enough),  by applying (\ref{error2}) and the total path length result in (\ref{TPL}) we get \begin{multline}\label{nicesum3}
  X(T^n)=\sum_{d(v)=L}\frac {\alpha n_v}{\mu^{-1}{\ln{n_v}}}
-\frac {1}{\mu^{-1}\ln{n}} \sum_{d(v)\leq L}\alpha n_v e^{-(\mu^{-1}\ln{n})\lambda _{v}}-\frac{\zeta n}{\mu^{-2}\ln^2{n}}\\
 +\frac {\alpha n\sigma^2}{\ln^2{n}}+
o_p\Big(\frac {n}{{\ln^2{n}}}\Big).
\end{multline}

As in \cite{jan1} and \cite{holmgren} the proof of Theorem \ref{thm}, i.e., the main theorem, will be completed by a classical theorem 
for convergence of triangular arrays to infinitely divisible distributions, see e.g., \cite[Theorem 15.28]{kallenberg}. 
First we recall the definition of \begin{align}\label{xi}\xi_v:=\frac{mn_v}{n}e^{-m\lambda_v}                                  \end{align}
in Section \ref{notation}.
Normalizing $X(T^n)$ gives by using (\ref{nicesum3}),
\begin{align}\label{mainthm4}&\frac{\mu^{-2}\ln^2{n}}{\alpha n}\left(X(T^n)-\frac{\alpha n}{\mu^{-1}\ln{n}} 
-\frac{\alpha n\ln\ln{n}}{\mu^{-1}\ln^2{n}}+\frac{\zeta n}{\mu^{-2}\ln^2{n}}\right)\nonumber\\&=-\sum_{d(v)\leq L}\xi_v+\frac{\mu^{-2}\ln^2{n}}{n}\sum_{d(v)= L}
\frac{n_v}{\mu^{-1}\ln{n_v}}-\mu^{-1}\ln\ln{n}\nonumber\\&~~~~~~~~~~~~~-\mu^{-1}\ln{n}+\mu^{-2}\sigma^2+o_p(1).
\end{align} 

Let \begin{align} \label{D} D:=\frac{\mu^{-2}\ln^2{ n}}{ n}\sum_{d(v)=L}\frac{ n_v}{\mu^{-1}\ln{ n_v}}-\mu^{-1}\ln\ln{ n}-\mu^{-1}\ln{ n}+\mu^{-2}\sigma^2.
\end{align}
and $\xi_i^{'}=\frac{-D}{n}$.
Thus,
\begin{align}\label{mainthm4,1}&\frac{\mu^{-2}\ln^2{n}}{\alpha n}\left(X(T^n)-\frac{\alpha n}{\mu^{-1}\ln{n}} 
-\frac{\alpha n\ln\ln{n}}{\mu^{-1}\ln^2{n}}+\frac{\zeta n}{\mu^{-2}\ln^2{n}}\right)\nonumber\\&=-\sum_{d(v)\leq L}\xi_v-\sum_{i=1}^{n}\xi_i^{'}+o_p(1).
\end{align}

As in \cite{holmgren} since the $n_v$'s in the sums in (\ref{mainthm4}) are not independent 
(although they are less dependent for vertices that are far from each other),
$\{\xi_v\}\bigcup \{\xi_i^{'}\}$ is not a triangular array. Recall the definition of $\Omega_L$ as the $\sigma$-field generated by $\{n_v,~d(v)\leq L\}$.
Hence, conditioned on $\Omega_L$, $\{\xi_v\}\bigcup \{\xi_i^{'}\}$ is a triangular array with $\xi_i^{'}$ conditioned on $\Omega_L$ deterministic.

\subsection{ Applying a limit theorem for sums of triangular arrays}

\subsubsection{Theorem \ref{thm3} which proves Theorem \ref{thm}}
As in \cite{jan1} and \cite{holmgren}, the proof of Theorem \ref{thm} will be completed by a classical theorem 
for convergence of sums of triangular arrays to infinitely divisible distributions, see e.g., \cite[Theorem 15.28]{kallenberg}. 
 For the sake of independence we intend to
condition on the $n_v$'s in the sums in (\ref{mainthm4,1}).
We show that conditioned on the $n_v$'s we get convergence in distribution for the normalized $X(T^n)$ to a random variable $W$ with an infinitely divisible distribution, which is not depending on the $n_v$'s we conditioned on. Then it follows in the same way as in \cite{holmgren} that also unconditioned the normalized $X(T^n)$ 
 converges in distribution to $W$.
The main Theorem \ref{thm} is proven by Theorem \ref{thm3} below.

\begin{thm}\label{thm3}
Choose any constant $c>0$ and let $n\rightarrow\infty$. 
Conditioning on the $\sigma$-field  $\Omega_L$, where $L=\lfloor \beta \log_b\ln n\rfloor$, if the constant $\beta$ is chosen large enough the following hold:
\begin{align*}&(i)~~\sup_v\mathbf{P}\big(\xi_v>x\big|\Omega_L\big)\longrightarrow 0~~\text{for~every}~x>0,\\&
(ii)~~\Delta_1:=\sum_{d(v)\leq L}\mathbf{P}\big(\xi_v>x\big|\Omega_L\big) \stackrel{p}\longrightarrow \nu(x,\infty)=\frac{\mu^{-1}}{x}
~~\text{for~every}~x>0,\\&(iii)~~\Delta_2:=\sum_{d(v)\leq L}\mathbf{E}\big(\xi_v\mathbf{1} [ \xi_v\leq c ]\big|\Omega_L \big)-
\frac{\mu^{-2}\ln^2{n}}{ n}\sum_{d(v)= L}\frac{ n_v}{\mu^{-1}\ln{n_v}}+\mu^{-1}\ln\ln{n}\\&~~~~+\mu^{-1}\ln{n}-\mu^{-2}\sigma^2\\&~~~~~~~~~~~~~~~~~~~~~~~~~~~\stackrel{p}
\longrightarrow-\mu^{-1}\ln\mu^{-1}+\mu^{-1}-\mu^{-2}\sigma^2-\frac{\sigma^2-\mu^2}{2\mu^2}+\mu^{-1}\ln{c},\\&(iv)~~
\Delta_3:=\sum_{d(v)\leq L}\mathbf{Var}\big(\xi_v\mathbf{1} [ \xi_v\leq c ]
\big|\Omega_L\big)
 \stackrel{p}\longrightarrow \mu^{-1}c.\end{align*}
\end{thm}

Before proving Theorem \ref{thm3} we will show how it proves Theorem \ref{thm}. 
Recall from (\ref{D}) that
\begin{align*} D=\frac{\mu^{-2}\ln^2{ n}}{ n}\sum_{d(v)=L}\frac{ n_v}{\mu^{-1}\ln{ n_v}}-\mu^{-1}\ln\ln{ n}-\mu^{-1}\ln{ n}+\mu^{-2}\sigma^2.
\end{align*}
We apply \cite[Theorem 15.28]{kallenberg}  with 
\begin{align} \label {b} a=0,~ ~~b=-\mu^{-1}\ln\mu^{-1}+\mu^{-1}-\mu^{-2}\sigma^2-\frac{\sigma^2-\mu^2}{2\mu^2}
\end{align} to $\sum_{d(v)\leq L}\xi_v+\sum_{i=1}^{n}\xi_i^{'}$ conditioned on $\Omega_L$ with  $\xi_i^{'}=\frac{-D}{n}$ deterministic. 
The constants $a$ and $b$ are the constants that 
occur in the general formula of the characteristic function for infinitely divisible distributions 
in (\ref{characteristic}).
Note that $\frac{D}{n}\rightarrow 0$, thus because of $(i)$, conditioned on $\Omega_L$, $\{\xi_v\}\bigcup \{\xi_i^{'}\}$ is a null array. 

We define $S(n):=\sum_{d(v)\leq L}\xi_v+\sum_{i=1}^{n}\xi_i^{'}$.
From ($ii$) we have that $\frac{d\nu}{dx}=\frac{\mu^{-1}}{x^2}$, hence
\begin{align*}&\int_{0}^{c}x^2 d\nu(x)=\int_{0}^{c}\mu^{-1}dx=\mu^{-1}c~\mathrm{and}~\int_{c}^{1}x d\nu(x)=\int_{c}^{1}\frac{\mu^{-1}}{x}dx=-\mu^{-1}\ln{c}.
\end{align*}
Thus, the right hand-sides of (\textit{iii}) and (\textit{iv}) are $b-\int_{c}^{1}x d\nu(x)$ and $\int_{0}^{c}x^2 d\nu(x)$, respectively,  where $b$ is the constant in (\ref{b}). 
The convergence in Theorem \ref{thm3} is in the probabilistic sense, while \cite[Theorem 15.28]{kallenberg} requires 
usual convergence, i.e., standard point-wise convergence of sequences with no probability involved. However, if the convergence instead were a.s.\ in Theorem \ref{thm3}, then it would have been easy to see from this theorem
that conditionally on $\Omega_L$ the conditions of \cite[Theorem 15.28]{kallenberg} are fullfilled for $S(n)$. Thus, assuming a.s.\ convergence in Theorem \ref{thm3}, \cite[Theorem 15.28]{kallenberg} implies that conditioned on $\Omega_L$,
\begin{align}\label{subsub1}&
S(n) \stackrel{d}\rightarrow W,~~as~n\rightarrow \infty,\end{align} where $W$ has an infinitely divisible distribution (in particular a weakly 1-stable distribution in this case)  with characteristic function 
\begin{align*}\mathbf{E}\Big(e^{itW}\Big)=\exp\left(itb+
\int_{0}^{\infty}(e^{itx}-1-itx\mathbf{1}[x<1])d\nu(x)\right);
\end{align*}
this is (\ref{mainthm2}) in Remark \ref{rem3} (since $b=C+\mu^{-1}(\gamma -1)$) which can be simplified to (\ref{mainthm3}) in Theorem \ref{thm}.

It follows from (\ref{subsub1}) that conditioning on $\Omega_L$ has no influence on the distributional convergence of $S(n)$ (unconditioned), since for any continuous bounded function $g:\mathbb R\rightarrow\mathbb R$,
\begin{align*} \mathbf{E}\Big(g(S(n))\mid \Omega_L\Big)=\int gdF(S(n)\mid \Omega_L))\stackrel{n\rightarrow \infty}\longrightarrow \mathbf{E}\Big(g(W)\Big).\end{align*} 
Thus, taking expectation by dominated convergence \begin{align*}\mathbf{E}\Big(g(S(n))\Big)\stackrel{n\rightarrow \infty}\longrightarrow\mathbf{E}\Big(g(W)\Big).\end{align*}
 This shows that also unconditioned $S(n) \stackrel{d}\rightarrow W$.
 Thus, unconditioned the normalized $X(T^n)$ in (\ref{mainthm4}) converges in distribution to $-W$.

It remains to show that convergence in probability (which is the type of convergence in Theorem \ref{thm3}) actually is sufficient for  $S(n) \stackrel{d}\rightarrow W$ to hold. In \cite{holmgren} we proved this fact for the binary search tree in two ways, in one by using subsequences and in the other one by using Skorohod's coupling theorem, see e.g., \cite[Theorem 3.30]{kallenberg}. By analogy these proofs also work for general split trees.
Thus, the proof of Theorem \ref{thm} for $X_v(T)$ is completed. 

Now it follows easily, by the same type of argument as for the binary search tree \cite{holmgren} that the result holds for $X_e(T)$ too. One way to see this is to consider $\widehat{T}$ as the tree $T$ with the root deleted. Then there is a natural 1-1 correspondence between edges of $T$
and vertices of $\widehat{T}$, and this correspondence also preserves the record (and cutting)
 operations. Since it is very unlikely that the root value would decide if values at high levels are records or not, it follows that asymptotically $X_e(T)$ and $X_v(T)$ have the same distribution.
Thus, the proof of Theorem \ref{thm} is completed.

The idea of the proof of Theorem \ref{thm3} is as for the binary search tree \cite[Theorem 2.1]{holmgren} to use Chebyshev's inequality to prove $(ii)$, $(iii)$ and $(iv)$ of Theorem \ref{thm3} ($(i)$ is very easy to prove). For the binary search tree  we frequently used in \cite[Theorem 2.1]{holmgren} that the sum $\sum_{r=1}^{k}\ln U_r$, where  $U_r$ are uniform $U(0,1)$ random variables, is distributed as a $-\Gamma(k,1)$ random variable. For general split trees, the solution of the renewal function $U(t)$ in (\ref{renewal equation3}) is fundamental for the proof of Theorem \ref{thm3}.

\subsubsection{Lemmas for the Proof of Theorem \ref{thm3}}

Recall that we write $\Omega_j$ for the $\sigma$-field generated by $\{n_v,~d(v)\leq j\}$ and $\mathscr{G}_j$ for
the $\sigma$-field generated by 
$\{W_{r,v},~r\in \{1,2\dots,j\}\}$, for all vertices $v$ with $d(v)=j$. Also recall that we write $L=\lfloor \beta \log_b\ln n\rfloor$.
We also write \begin{align}\label{nice}\widehat{n_v}:= n\prod_{r=1}^{k}W_{r,v},~~\mathrm{and}~~\widehat{\xi_v}:=\frac{m\widehat{n_v}}{n}e^{-m\lambda_v},
\end{align} 
where $m:=\mu^{-1}\ln{n}$.
Note that $\mathscr{G}_j$ is thus equivalently the $\sigma$-field generated by $\{\widehat{n_v}:d(v)\leq j\}$.

We present below four crucial lemmas by which we can then easily prove Theorem \ref{thm3}.

\begin{Lemma}\label{lem5}Suppose that $n\rightarrow\infty$ and choose any constant $c>0$. Then for $L=\lfloor \beta \log_b\ln n\rfloor$ and $\beta$ large enough, the following hold
 \begin{align*}& 
\sum_{d(v)\leq L}\mathbf{P}\big(\xi_v>x\big|\Omega_L\big)=\sum_{d(v)\leq L}
\mathbf{P}\big(\widehat{\xi_v}>x\big|\mathscr G_{L}\big)+o_p(1),\\& \sum_{d(v)\leq L}\mathbf{E}\big(\xi_v\mathbf{1} 
[ \xi_v\leq c ]\big|\Omega_L\big)=\sum_{d(v)\leq L}\mathbf{E}\big(\widehat{\xi_v}\mathbf{1} 
[ \widehat{\xi_v}\leq c ]\big|\mathscr G_{L} \big)+o_p(1),\\&\sum_{d(v)= L}\frac{n_{v}}{\mu^{-1}\ln{n_v}}=
\frac{n}{\mu^{-1}\ln{n}}
-\sum_{d(v)=L}\frac{\widehat{n_v}\ln(\frac{\widehat{n_v}}{n})}{\mu^{-1}\ln^2{n}}+o_p(\frac{n}{\ln ^2 n}),
\\&\sum_{d(v)\leq L}\mathbf{Var}\big(\xi_v\mathbf{1} [ \xi_v\leq c ]\big|
\Omega_L\big)=\sum_{d(v)\leq L}\mathbf{Var}\big(\widehat{\xi_v}\mathbf{1} [ \widehat{\xi_v}\leq c ]\big|
\mathscr G_{L}\big)+o_p(1).
\end{align*}
\end{Lemma}

For simplicity we sometimes use a short notation for the following sums, i.e., \begin{align}
\Phi_v:&=\left(\frac{n}{\mu^{-1}\ln{n}}
-\sum_{d(v)=L}\frac{\widehat{n_v}\ln(\frac{\widehat{n_v}}{n})}{\mu^{-1}\ln^2{n}}\right),\nonumber\\
\label{R1} R_1:&=\sum_{d(v)\leq L}\mathbf{P}\big(\widehat{\xi_v}>x\big|\mathscr G_{L}\big),\\
\label{R2}R_2:&=\sum_{d(v)\leq L}\mathbf{E}\big(\widehat{\xi_v}\mathbf{1} 
[ \widehat{\xi_v}\leq c ]\big|\mathscr G_{L} \big)-\frac{\mu^{-2}\ln^2{n}}{n}\cdot \Phi_v,\\
\label{R3}R_3:&=\sum_{d(v)\leq L}\mathbf{Var}\big(\widehat{\xi_v}\mathbf{1} [ \widehat{\xi_v}\leq c ]\big|
\mathscr G_{L}\big).
\end{align} 

\begin{Lemma}\label{lem6}
Suppose that $n\rightarrow\infty$ and choose any constant $c>0$. Then for $L=\lfloor \beta \log_b\ln n\rfloor$ and $\beta$ large enough, the following hold
 \begin{align*}&
\mathbf{E}\Big(R_1\Big)=\frac{\mu^-1}{x}+o(1)=\nu(x,\infty)+o(1),
~~\text{for~every}~x>0,\\& \mathbf{E}\Big(R_2\Big)=-\mu^{-1}\ln{n}-\mu^{-1}\ln\ln{n}+\mu^{-1}-\mu^{-1}\ln\mu^{-1}+\mu^{-1}\ln{c}-
\frac{\sigma^2-\mu^2}{2\mu^2}+o(1),\\&\mathbf{E}\Big(R_3\Big)=\mu^{-1}c+o(1).
\end{align*}
\end{Lemma}

Let $l:=\lfloor \frac{\log_b\ln{n}}{2}\rfloor$ and for short write  \begin{align}\label{S1} S_1:&=\sum_{l\leq d(v)\leq L}
\mathbf{P}\big(\widehat{\xi_v}>x\big|\mathscr G_{L}\big),\\\label{S2} 
S_2:&=\sum_{{l\leq d(v)\leq L}}\mathbf{E}\big(\widehat{\xi_v}\mathbf{1} 
[ \widehat{\xi_v}\leq c ]\big|\mathscr G_{L} \big)- \frac{\mu^{-2}\ln^2{n}}{n}\cdot\Phi_v,\\\label{S3} 
S_3:&=\sum_{{l\leq d(v)\leq L}}\mathbf{Var}\big(\widehat{\xi_v}\mathbf{1} [ \widehat{\xi_v}\leq c ]\big|
\mathscr G_{L}\big).
\end{align} 

\begin{Lemma}\label{lem7} Suppose that $n\rightarrow\infty$. Then for $L=\lfloor \beta \log_b\ln n\rfloor$ (where $\beta$ is large enough) and $l=\lfloor \frac{\log_b\ln{n}}{2}\rfloor$, the following limits hold  
\begin{align}\label{tjorn}&\mathbf{Var}\left(\mathbf{E}\Big(S_1\Big|\mathscr G_{l}\Big)\right)\rightarrow 0,
\\\label{lightout1}&\mathbf{Var}\left(\mathbf{E}\Big(S_2
\Big|\mathscr G_{l}\Big)\right)\rightarrow 0,\\
\label{scar1}&\mathbf{Var}\left(\mathbf{E}\Big(S_3\Big| \mathscr G_{l}\Big)\right)\rightarrow 0.\end{align}

\end{Lemma}

\begin{Lemma}\label{lem8} Suppose that $n\rightarrow\infty$. Then for $L=\lfloor \beta \log_b\ln n\rfloor$ (where $\beta$ is large enough) and $l=\lfloor \frac{\log_b\ln{n}}{2}\rfloor$ the following limits hold
\begin{align}\label{cia1}&\mathbf{E}\left(\mathbf{Var}\Big(S_1\Big|\mathscr G_{l}\Big)\right)\rightarrow 0,\\\label{cia2}
&\mathbf{E}\left(\mathbf{Var}\Big(S_2
\Big|\mathscr G_{l}\Big)\right)\rightarrow 0,\\
\label{scarlet2}&\mathbf{E}\left(\mathbf{Var}\Big(S_3\Big| \mathscr G_{l}\Big)\right)\rightarrow 0.
\end{align} 
\end{Lemma}

Before proving these lemmas we show how their use leads to the proof of Theorem \ref{thm3}.

\subsubsection{Proof of Theorem \ref{thm3}}

 Recall that $m=\mu^{-1}\ln n$.
 For any $x>0$, and $v$ with $d(v)\leq L$, we have 
\begin{align}\label{sup} \mathbf{P}\big(\xi_v>x\big|\Omega_L\big)
 &=\mathbf{P}\big(e^{-m\lambda_v}>\frac{nx}{mn_v}\big|\Omega_L\big)=\mathbf{P}\big(\lambda_v<\frac{1}{m}\ln \frac{mn_v}{nx}\big|
\Omega_L\big)\nonumber\\&=1-\exp\big(-\frac{1}{m}\ln_+\frac{mn_v}{nx}\big). \end{align} 
Thus, for every $x>0$,
\begin{equation}\label{sup2} \mathbf{P}\big(\xi_v>x\big|\Omega_L\big)\leq \frac{1}{m}\ln_+\frac{mn_v}{nx}\leq\frac{1}{m}\ln_+\frac{m}{x}\rightarrow 0,\end{equation}
which proves (i).

Recall the definitions of $R_1$, $R_2$ and $R_3$ in (\ref{R1}), (\ref{R2}) and (\ref{R3}).
Note that Lemma \ref{lem5} shows that in Theorem \ref{thm3} the left hand-sides of 
$(ii)$, $(iii)$ and $(iv)$, i.e., $\Delta_1$, $\Delta_2$ and $\Delta_3$, respectively, are equal to 
\begin{align*}\Delta_1&=R_1+o_p(1),\\\Delta_2&=R_2+\mu^{-1}\ln\ln{n}+\mu^{-1}\ln{n}-\mu^{-2}\sigma^2+o_p(1):=\widehat{R}_2+o_p(1),\\ \Delta_3&=R_3+o_p(1).
\end{align*}
Lemma \ref{lem6} shows that 
the expected values of $R_1$, $\widehat{R}_2$ and $R_3$ converge to the right hand-sides in $(ii)$, $(iii)$ and $(iv)$ of Theorem \ref{thm3}.

We complete the proof of Theorem \ref{thm3} by showing that
\begin{align}\label{lind}&\mathbf{Var}\left(R_1\right) \rightarrow 
0~~\text{for~every}~x>0,~~~
\mathbf{Var}\left(R_2\right)
 \rightarrow 0,~~\mathrm{and}~~
\mathbf{Var}\left(R_3\right) \rightarrow 0.\end{align}
Then by Chebyshev's inequality $(ii)$, $(iii)$ and $(iv)$ of Theorem \ref{thm3} follow.
Thus, it remains to show how (\ref{lind}) follows from Lemma \ref{lem6} and Lemma \ref{lem7}.

 By using (\ref{sup2}), one easily obtains
\begin{align}\label{samesum1}\sum_{d(v)\leq L}\mathbf{P}\big(\widehat{\xi_v}>x\big|\mathscr G_{L}\big)
&=\sum_{l\leq d(v)\leq L}\mathbf{P}\big(\widehat{\xi_v}>x\big|\mathscr G_{L}\big)+o(1), \\\label{samesum2}
\sum_{d(v)\leq L}\mathbf{E}\big(\widehat{\xi_v}\mathbf{1} 
[ \widehat{\xi_v}\leq c ]\big|\mathscr G_{L} \big)&=
\sum_{l\leq d(v)\leq L}\mathbf{E}\big(\widehat{\xi_v}\mathbf{1} 
[ \widehat{\xi_v}\leq c ]\big|\mathscr G_{L} \big)+o(1),
\\\label{samesum3}\sum_{d(v)\leq L}\mathbf{Var}\big(\widehat{\xi_v}\mathbf{1} [ \widehat{\xi_v}\leq c ]\big|
\mathscr G_{L}\big)&=
\sum_{l\leq d(v)\leq L}\mathbf{Var}\big(\widehat{\xi_v}\mathbf{1} [ \widehat{\xi_v}\leq c ]\big|
\mathscr G_{L}\big)+o(1). 
\end{align}
Hence,
\begin{align}\label{haj}R_1=S_1+o(1),~~R_2=S_2+o(1),~~R_3=S_3+o(1).
\end{align}

To show (\ref{lind}) we use a variance formula that is easy to establish, 
see e.g., \cite[exercise 10.17-2]{gut}, 
\begin{align}\label{condvar} \mathbf{Var}(X)=\mathbf{E}(\mathbf{Var}(X\mid\mathscr{G}))+\mathbf{Var}(\mathbf{E}(X\mid\mathscr{G})),\end{align} where
$X$ is a random variable and $\mathscr{G}$ is a sub $\sigma$-field.

Recall that  $\mathscr{G}_j$ is the $\sigma$-field generated by $\{W_{r,v},~r\in\{1,2,\dots,j\}\}$, for all vertices with $d(v)=j$.
Consequently,  by applying the variance formula in (\ref{condvar}), from Lemma \ref{lem7} and Lemma \ref{lem8} we get as $n\rightarrow \infty$
\begin{align*}&\mathbf{Var}(S_1)=
\mathbf{E}(\mathbf{Var}(S_1|\mathscr G_{l}))+\mathbf{Var}(\mathbf{E}(S_1|\mathscr G_{l}))\rightarrow 0,
\\ &\mathbf{Var}(S_2)=\mathbf{E}\left(\mathbf{Var}(S_2|\mathscr G_{l})\right)
+\mathbf{Var}\left(\mathbf{E}(S_2|\mathscr G_{l})\right)\rightarrow 0,
\\ &\mathbf{Var}(S_3)=
\mathbf{E}(\mathbf{Var}(S_3|\mathscr G_{l}))+\mathbf{Var}(\mathbf{E}(S_3|\mathscr G_{l}))\rightarrow 0,
\end{align*}
and thus (\ref{lind}) follows from (\ref{haj}). 

\textbf{We have proved Theorem \ref{thm3} by the use of the lemmas, and thus also Theorem \ref{thm}.}

\subsubsection{Proofs of the Lemmas of Theorem \ref{thm3}}
Finally we present the proofs of Lemma \ref{lem5}, Lemma \ref{lem6}, Lemma \ref{lem7} and Lemma \ref{lem8}.
\begin{proof}[Proof of Lemma \ref{lem5}]


From (\ref{binomial}) and (\ref{binomial,1})  in Section \ref{Subtrees} we get in particular that given $\mathscr{G}_L$, 
\begin{align*}
n_v &\leq \mathrm{Binomial}(n,\prod_{r=1}^{k}W_{r,v})+s_1L,\\
n_v & >\mathrm{Binomial}(n,\prod_{r=1}^{k}W_{r,v})-sL.
\end{align*}
Since a Binomial $(k,p)$ random variable has expected value $kp$ and variance  $kp(1-p)$, the Chebyshev inequality results in
\begin{align}\label{iliada1}\mathbf{P}\big(\mid n_v-n\prod_{r=1}^{k}W_{r,v}\mid > n^{0.6}|\Omega_L\big)\leq \frac{1}{n^{0.19}}.\end{align}

 This motivates the notation of $\widehat{n_v}:= n\prod_{r=1}^{k}W_{r,v}$ in (\ref{nice}).
Also recall that we write $\widehat{\xi_v}:=\frac{m\widehat{n_v}}{n}e^{-m\lambda_v}$ for $m:=\mu^{-1}\ln{n}$, and that
 $\mathscr{G}_j$ is the $\sigma$-field generated by $\{\widehat{n_v}:d(v) \leq j\}$. By using (\ref{sup}) and (\ref{sup2}) we get (compare with \cite[equation (55)]{holmgren}), 
\begin{align}\label{jul1}\sum_{d(v)\leq L}\mathbf{P}\big(\xi_v>x\big|\Omega_L\big)&=\sum_{k=1}^{L}\sum_{d(v)=k}\frac{1}{m}\ln_+\frac{m{n_v}}{nx}\Big(1+\mathcal O \big(\frac{\ln m}{m}\big)\Big)
\end{align} 
and similarly 
\begin{align}\label{jul2}\sum_{d(v)\leq L}\mathbf{P}\big(\widehat{\xi_v}>x\big|\mathscr G_{L}\big)&=\sum_{k=1}^{L}\sum_{d(v)=k}\frac{1}{m}\ln_+\frac{m\widehat{n_v}}{nx}\Big(1+\mathcal O \big(\frac{\ln m}{m}\big)\Big).
\end{align} 
By using (\ref{jul1}), (\ref{jul2}) and (\ref{iliada1}) we get
\begin{align}\label{iliada2} \sum_{d(v)\leq L}\mathbf{P}\big(\xi_v>x\big|\Omega_L\big)=\sum_{d(v)\leq L}\mathbf{P}\big(\widehat{\xi_v}>x\big|\mathscr G_{L}\big)+o_p(1).
\end{align}
One easily gets (compare with \cite[p.251]{jan1} and 
\cite[equation (61)--(62)] {holmgren}) that 
\begin{align}\label{jul3}\sum_{d(v)\leq L}\mathbf{E}\big(\xi_v\mathbf{1} 
[ \xi_v\leq c ]\big|\Omega_L\big)&= \sum_{d(v)\leq L}\frac{mn_v}{n(m+1)}e^{-\frac{m+1}{m}\ln_+({\frac{mn_v}{nc}})}
\end{align}
and similarly 
\begin{align}\label{jul4}\sum_{d(v)\leq L}\mathbf{E}\big(\widehat{\xi_v}\mathbf{1} 
[ \widehat{\xi_v}\leq c ]\big|\mathscr G_{L} \big)&=\sum_{d(v)\leq L} \frac{m\widehat{n_v}}
{n(m+1)}e^{-\frac{m+1}{m}\ln_+({\frac{m\widehat{n_v}}{nc}})}.
\end{align}
Thus, (\ref{iliada1}) implies that 
\begin{align}\label{array1}\sum_{d(v)\leq L}\mathbf{E}\big(\xi_v\mathbf{1} 
[ \xi_v\leq c ]\big|\Omega_L\big)=\sum_{d(v)\leq L}\mathbf{E}\big(\widehat{\xi_v}\mathbf{1} 
[ \widehat{\xi_v}\leq c ]\big|\mathscr G_{L} \big)+o_p(1).\end{align}

Using the bound in (\ref{bound}) for the sum of the subtree sizes with
 $n_v$ less than $\frac{n}{b^{kL}}$ (for $k$ large enough) we get the expansion
\begin{align*}
\sum_{d(v)= L}\frac{n_{v}}{\mu^{-1}\ln{n_v}}=\frac{n}{\mu^{-1}\ln{n}}-\sum_{d(v)=L}\frac{n_{v}\ln(\frac{n_{v}}{n})}{\mu^{-1}\ln^2{n}}+o(\frac{n}{\ln^2{n}}).
\end{align*}

By again using (\ref{iliada1})  (compare with \cite[equation (68)] {holmgren}) we get
\begin{align*}\sum_{d(v)= L}\frac{n_{v}}{\mu^{-1}\ln{n_v}}=\frac{n}{\mu^{-1}\ln{n}}
-\sum_{d(v)=L}\frac{\widehat{n_v}\ln(\frac{\widehat{n_v}}{n})}{\mu^{-1}\ln^2{n}}+o_p(\frac{n}{\ln^2{n}}).
\end{align*}

By using the calculations in \cite[p.251-252]{jan1} (compare with \cite[equation (70)] {holmgren}) we get
\begin{align}\label{skoj}\sum_{d(v)\leq L}\mathbf{Var}\big(\xi_v\mathbf{1} [ \xi_v\leq c ]\big|
\Omega_L\big)&=\sum_{d(v)\leq L}\frac{m^2{n_v}^2}{2mn^2}e^{-\frac{2m+1}{m}\ln_+({\frac{mn_v}{nc}})}+o(1),
\end{align}
and similarly 
\begin{align}\label{jul5}\sum_{d(v)\leq L}\mathbf{Var}\big(\widehat{\xi_v}\mathbf{1} [ \widehat{\xi_v}\leq c ]\big|
\mathscr G_{L}\big)=\sum_{d(v)\leq L}\frac{m^2{\widehat{n_v}}^2}{2mn^2}
e^{-\frac{2m+1}{m}\ln_+({\frac{m\widehat{n_v}}{nc}})}+o(1)
\end{align}
Thus, using (\ref{iliada1}) we obtain 
\begin{align}\label{skoj2}\sum_{d(v)\leq L}\mathbf{Var}\big(\xi_v\mathbf{1} [ \xi_v\leq c ]\big|
\Omega_L\big)=\sum_{d(v)\leq L}\mathbf{Var}\big(\widehat{\xi_v}\mathbf{1} [ \widehat{\xi_v}\leq c ]\big|
\mathscr G_{L}\big)+o_p(1).
\end{align}
\end{proof}


\begin{proof}[Proof of Lemma \ref{lem6}]
Recall that we write
\begin{align}\label{Y}
 Y_k=-\sum_{r=1}^{k}\ln W_r
\end{align}
and that we write 
\begin{align*}R_1=\sum_{d(v)\leq L}\mathbf{P}\big(\widehat{\xi_v}>x\big|\mathscr G_{L}\big).
\end{align*}
As in the calculations in \cite[equations (56)]{holmgren} from (\ref{jul2}) one gets 
\begin{align}\label{meas2}\mathbf{E}(R_1)=(1+o(1))\sum_{k=1}^{L} b^k \mathbf{E}\Big{(}\frac{(\ln{m}-\ln{x}-Y_k)}{m}I\{Y_k\leq \ln{m} -\ln{x}\}\Big{)}.
\end{align}
By using integration by parts we get that the sum in (\ref{meas2}) is equal to
\begin{align}\label{meas3}&\sum_{k=1}^{L} b^k\frac{1}{m}\int_{0}^{\ln{m}-\ln{x}} 
 \mathbf{P}(Y_k\leq t)dt= \frac{1}{m}\int_{0}^{\ln{m}-\ln{x}}\sum_{k=1}^{L} b^k 
 \mathbf{P}(Y_k\leq t)dt. \end{align}

Recall the definition of the renewal function $U(t):=\sum_{k=1}^{\infty}b^k\mathbf{P}(Y_k\leq t)$ 
in (\ref{renewal function}) above.
We want to show that
\begin{align}\label{ratio2}&\frac{1}{m}\int_{0}^{\ln{m}-\ln{x}}\sum_{k=L+1}^{\infty} b^k 
 \mathbf{P}(Y_k\leq t)dt=o(1).
\end{align}
To show this we use large deviations.  Choose an arbitrary $s>0$, by applying the Markov inequality and using that 
the $W_{r,v}$, $r\in\{1,\dots,k\}$, are i.i.d.\ we get 
\begin{align}\label{largedev}&
 \mathbf{P}(Y_k\leq t)= \mathbf{P}(-Y_k\geq -t)=\mathbf{P}(e^{-sY_k} \geq e^{-st})\leq \Big(\mathbf{E}(V^s)\Big)^k e^{st}.
\end{align}
 Choosing $s>1$, we get
\begin{align*}
 \mathbf{E}(V^{s})<\mathbf{E}(V)=\frac{1}{b}.
\end{align*}
 Thus, we can find $\delta>0$ such that 
\begin{align}\label{largedev2}
 \mathbf{E}(V^{s})\leq\frac{1}{b^{1+\delta}}.
\end{align}
In the definition of $L=\lfloor \beta \log_b\ln{n}\rfloor$ the constant $\beta$ can be chosen arbitrarily large. It is enough to show that $\frac{\ln m}{m}\sum_{k=L+1}^{\infty}b^k\mathbf{P}(Y_k\leq \ln{m}-\ln{x})$ is $o(1)$ for proving (\ref{ratio2}). 
By applying (\ref{largedev}) and (\ref{largedev2}) we get that
\begin{align}\label{largedev3}
\sum_{k=L+1}^{\infty} b^k \mathbf{P}(Y_k\leq \ln{m}-\ln{x})&\leq\sum_{k=L+1}^{\infty}\frac{b^k}{b^{k+\delta k}}\frac{m^{s}}{x^{s}}=\sum_{k=L+1}^{\infty}b^{-\delta k}\frac{m^{s}}{x^{s}}\nonumber \\&=\mathcal{O}\Big(m^{-\delta\beta}m^{s}\Big).
\end{align}
Thus, choosing $\beta>\frac{s-1}{\delta}$ in $L$ gives (\ref{ratio2}).
Now the solution of $U(t)$ in (\ref{renewal equation3}) gives that the quantity in (\ref{meas3}) is equal to
\begin{align}\label{meas3,2}\int_{0}^{\ln{m}-\ln{x}}U(t)dt+o(1)&=\frac{\mu^{-1}+o(1)}{m}\int_{0}^{\ln{m}-\ln{x}} e^t dt+o(1)
\nonumber\\&= \frac{\mu^{-1}}{x}+o(1)=\nu (x,\infty)+o(1).
\end{align}
Hence, $\mathbf{E}(R_1)=\nu(x,\infty)+o(1)$.

In analogy with \cite[equation (63)] {holmgren}) by using (\ref{jul4}) and (\ref{Y}) we deduce that
\begin{align}\label{array2} &\mathbf{E}\Big(\sum_{d(v)\leq L}\mathbf{E}\big(\widehat{\xi_v}\mathbf{1} 
[ \widehat{\xi_v}\leq c ]\big|\mathscr G_{L} \big)\Big):=E_1+E_2,
\end{align}
where
\begin{align}\label{ekvtill}E_1&=\mathbf{E}\sum_{d(v)\leq L}\frac{m}{m+1}e^{-Y_k}e^{-\frac{m+1}{m}(\ln{m}-\ln{c}-Y_k)}
I\{Y_k\leq \ln{m}-\ln{c}\},\nonumber\\E_2&=\mathbf{E}\sum_{d(v)\leq L}\frac{m}{m+1}e^{-Y_k}I\{Y_k>\ln{m}-\ln{c}\}.
\end{align}

By using integration by parts, applying the solution of $U(t)$ in (\ref{renewal equation3}) and using (\ref{largedev3}) we obtain that
 \begin{align}\label{1}E_1&=e^{-\frac{m+1}{m}(\ln{m}-\ln{c})}\frac{m}{m+1}\sum_{k=1}^{L}  b^k
 \int_{0}^{\ln{m}-\ln{c}}e^{\frac{t}{m}}d\mathbf{P}(Y_k\leq t)\nonumber 
\\&=e^{-\frac{m+1}{m}(\ln{m}-\ln{c})}\frac{m}{m+1}\bigg(\Big{|}\sum_{k=1}^{L}  b^ke^{\frac{t}{m}}\mathbf{P}(Y_k\leq t)\Big{|}_{0}^{\ln{m}-\ln{c}}\nonumber 
\\&~~~~~~~~~~~~~-\int_{0}^{\ln{m}-\ln{c}}\sum_{k=1}^{L}  \frac{b^k}{m}e^{\frac{t}{m}}\mathbf{P}(Y_k\leq t)dt\bigg)
\nonumber 
\\&=\mu^{-1}+o(1).\end{align}
By similar calculations as in (\ref{1}),
\begin{align}\label{array4}E_2&=\frac{m}{m+1}L-\frac{m}{m+1}\int_{0}^{\ln{m}-\ln{c}} \sum_{k=1}^{L} b^k e^{-t}d\mathbf{P}(Y_k\leq t)
\nonumber 
\\&=\frac{m}{m+1}L-\frac{m}{m+1}\bigg(\Big{|}\sum_{k=1}^{L}  b^ke^{-t}\mathbf{P}(Y_k\leq t)\Big{|}_{0}^{\ln{m}-\ln{c}}\nonumber 
\\&~~~~~~~~~~~~~+\int_{0}^{\ln{m}-\ln{c}}\sum_{k=1}^{L} b^k e^{-t}\mathbf{P}(Y_k\leq t)dt\bigg)
\nonumber \\&=L-\mu^{-1}-\frac{m}{m+1}\int_{0}^{\ln{m}-\ln{c}}\sum_{k=1}^{L} b^k e^{-t}\mathbf{P}(Y_k\leq t)dt+o(1).
\end{align}
From (\ref{largedev3}) it follows that
\begin{align*}&
 \int_{0}^{\ln{m}-\ln{c}}\sum_{k=1}^{L} b^k e^{-t}\mathbf{P}(Y_k\leq t)dt=\int_{0}^{\ln{m}-\ln{c}}e^{-t}U(t)dt+o(1)
\nonumber\\&=\int_{0}^{\ln{m}-\ln{c}}e^{-t}(U(t)-\mu^{-1}e^t)dt
+\mu^{-1}(\ln{m}-\ln{c})+o(1).
\end{align*}
Applying the solution of $W(x):=\int_{0}^{x}e^{-t}(U(t)-\mu^{-1}e^t)dt$ in (\ref{V}), from (\ref{array4}) we get that
\begin{align} \label{ekvtill2}E_2 &=L-\mu^{-1}\ln{m}+\mu^{-1}\ln{c}-\frac{\sigma^2-\mu^2}{2\mu^2}+o(1).\end{align}
Recalling (\ref{array2}) and  applying the approximations of $E_1$ in (\ref{1}) and $E_2$ in (\ref{ekvtill2}) we deduce that \begin{align*} &\mathbf{E}\Big(\sum_{d(v)\leq L}\mathbf{E}\big(\widehat{\xi_v}\mathbf{1} 
[ \widehat{\xi_v}\leq c ]\big|\mathscr G_{L} \big)\Big)=L+\mu^{-1}-\mu^{-1}\ln{m}+\mu^{-1}\ln{c}-\frac{\sigma^2-\mu^2}{2\mu^2}+o(1)\end{align*}
which is equal to
 \begin{align}\label{K}
  K:&=L+\mu^{-1}-\mu^{-1}\ln\ln{n}-\mu^{-1}\ln\mu^{-1}+\mu^{-1}\ln{c}-\frac{\sigma^2-\mu^2}{2\mu^2}+o(1).
 \end{align}

By the definition of $\widehat{n_v}$ in (\ref{nice}),
\begin{align}\label{exptot1,1}&\Phi_v:=\frac{n}{\mu^{-1}\ln{n}}
-\sum_{d(v)=L}\frac{\widehat{n_v}\ln(\frac{\widehat{n_v}}{n})}{\mu^{-1}\ln^2{n}}\nonumber\\&=\frac{n}{\mu^{-1}\ln{n}}
-\sum_{d(v)=L}\frac{n\prod_{r=1}^{L}W_{r,v}\sum_{r=1}^{L}
\ln{W_{r,v}}}{\mu^{-1}\ln^2{n}},
\end{align}
Hence, by using the definition of $\mu$ in (\ref{splitdef}) we get that
\begin{align}\label{exptot2}\mathbf{E}\Big(\Phi_v\Big)= \frac{n}{\mu^{-1}\ln{n}}+\frac{nL}{\mu^{-2}\ln^2{n}}.
\end{align}
Thus, recalling the definition of $R_2$ in (\ref{R2}) we get 
$\mathbf{E}\Big(R_2\Big)=K-\mu^{-1}\ln{n}-L$, where $K$ is defined in (\ref{K}). 

Recall that 
\begin{align*}&R_3=\sum_{d(v)\leq L}\mathbf{Var}\big(\widehat{\xi_v}\mathbf{1} [ \widehat{\xi_v}\leq c ]\big|
\mathscr G_{L}\big).
\end{align*}
 By using (\ref{jul5}) we get that
\begin{align}\label{array6}&\mathbf{E}\left(R_3\right)=\sum_{k=1}^{L}\frac{b^km}{2}\mathbf{E}\left(\prod_{r=1}^{k}W_{r,v}^2
e^{-\frac{2m+1}{m}(\ln{m}-\ln{c}-Y_k)I\{Y_k\leq \ln{m}-\ln{c}\}}\right)+o(1)\nonumber\\&=V_1+V_2+o(1),
\end{align}
where 
\begin{align}\label{V1}V_1&:=e^{-\frac{2m+1}{m}(\ln{m}-\ln{c})}\frac{m}{2}\int_{0}^{\ln{m}-\ln{c}}\sum_{k=1}^{L}b^k e^{\frac{t}{m}}d\mathbf{P}(Y_k\leq t),\nonumber\\V_2&:=\mathbf{E}\left(\sum_{k=1}^{L}\frac{b^km}{2}\prod_{r=1}^{k}W_{r,v}^2
I\{Y_k> \ln{m}-\ln{c} \}\right).
\end{align}

 By applying the solution of $U(t)$ in (\ref{renewal equation3}), integration by parts results in
\begin{align} \label{pelle3} V_2&=\int_{\ln{m}-\ln{c}}^{\infty}
\sum_{k=1}^{L}\frac{b^km}{2}e^{-2t}d\mathbf{P}(Y_k\leq t)\nonumber\\&=\frac{m}{2}\Big|e^{-2t}U(t)\Big|_{\ln{m}-\ln{c}}^{\infty}+m\int_{\ln{m}-\ln{c}}
^{\infty}e^{-2t}U(t)dt+o(1)=
\frac{\mu^{-1}c}{2}+o(1),\end{align} where we used (\ref{largedev}) (choosing $1<s<2$) and then similar calculations as in (\ref{largedev3}) to show that if we sum over all $k$ instead of $k\leq L$ the error term is just $o(1)$.

Similarly, by using  (\ref{largedev3}), integration by parts gives
\begin{align} \label{pelle2} V_1&=\frac{\mu^{-1}c}{2}+o(1).
\end{align}
\end{proof}
Hence, $\mathbf{E}(R_3)=\mu^{-1}c+o(1)$.

\begin{proof}[Proof of Lemma \ref{lem7}]



For a given vertex $v_{i}\in T$ with $d(v_i)=l$, there are at most $b^{j-l}$ choices of  $v$ at depth $j$ with ancestor 
$v_{i}$. Recall that $Y_{j,{v}}:=-\sum_{r=1}^{j}\ln{W_{r,v}}$.
 For $v$ with $d(v)=j$,  we also write
\begin{align}\label{ny}
  Z_{j-l,{v}}:=Y_{j,{v}}-Y_{l,{v_i}}=-\sum_{r=l+1}^{j}\ln{W_{r,v}}.
\end{align}

Recall from (\ref{S1}) that \begin{align*}S_1&=\sum_{l\leq d(v)\leq L}
\mathbf{P}\big(\widehat{\xi_v}>x\big|\mathscr G_{L}\big).
\end{align*}

  Using (\ref{samesum1}) and  the solution of the renewal equation $U(t)$ in (\ref{renewal equation3}) we get by similar calculations as in (\ref{meas2})--(\ref{meas3,2}),
 \begin{align}\label{array10}\mathbf{E}\Big(S_1|\mathscr G_{l}\Big)&= \sum_{i=1}^{b^l} \frac{1}{m}\int_{0}^{\ln{m}-\ln{x}-Y_{l,{v_i}}}\sum_{j=l+1}^{L} b^{j-l} \mathbf{P}(Z_{j-l,v}\leq t)dt+o(1)
 \nonumber \\&=\sum_{i=1}^{b^l}\frac{1}{m}\int_{0}^{\ln{m}-\ln{x}-Y_{l,{v_i}}}\mu ^{-1}e^t dt +o(1)\nonumber \\&=
 \sum_{i=1}^{b^l}\prod_{r=1}^{l}W_{r,v_i}\frac{\mu ^{-1}}{x}+o(1)=\frac{\mu ^{-1}}{x}+o(1).
\end{align}
Thus,
$\mathbf{Var}\Big(\mathbf{E}(S_1|\mathscr G_{l})\Big)$ is o(1),
which shows (\ref{tjorn}).

We  show that (\ref{lightout1}) is true by similar calculations as for showing (\ref{tjorn}).
Recall from (\ref{S2}) that \begin{align*}
             S_2&=\sum_{{l\leq d(v)\leq L}}\mathbf{E}\big(\widehat{\xi_v}\mathbf{1} 
[ \widehat{\xi_v}\leq c ]\big|\mathscr G_{L} \big)- \frac{\mu^{-2}\ln^2{n}}{n}\cdot \Phi_v,
            \end{align*}
where \begin{align*}&\Phi_v=\frac{n}{\mu^{-1}\ln{n}}
-\sum_{d(v)=L}\frac{n\prod_{r=1}^{L}W_{r,v}\sum_{r=1}^{L}
\ln{W_{r,v}}}{\mu^{-1}\ln^2{n}}.
\end{align*}

First, as before we let $v_{i}$ with $d(v_i)=l$ be a given vertex  so that there are at most $b^{j-l}$ choices of  $v$ 
at depth $j$ with ancestor $v_{i}$. 
Recall the notation of $Z_{j-l,v}$ in (\ref{ny}), i.e., $Y_{j,v}=Y_{l,v_i}+Z_{j-l,v}$.
By similar calculations as in (\ref{array10}) (glancing at the calculations in (\ref{array2})) we obtain
\begin{align*}\mathbf{E}\left(\sum_{l\leq d(v)\leq L}\mathbf{E}\big(\widehat{\xi_v}\mathbf{1}
 [ \widehat{\xi_v}\leq c ]\big|\mathscr G_{L} \big)\Big|\mathscr G_{l}\right) =F_1+F_2,
\end{align*}
where
\begin{multline*}F_1:=\mathbf{E}\bigg(\sum_{l \leq d(v)\leq L}
 \frac{m}{m+1}e^{-Y_{l,{v_i}}-Z_{j-l,{v}}}e^{-\frac{m+1}{m}(\ln{m}-
\ln{c}-Y_{l,{v_i}}-Z_{j-l,{v}})}\cdot\\I\{Y_{l,v_i}+Z_{j-l,v}\leq\ln{m}-\ln{c} \}\Big|\mathscr G_{l}\bigg),\\F_2:=
\mathbf{E}\left(\sum_{l \leq d(v)\leq L}\frac{m}{m+1}e^{-Y_{l,v_i}-Z_{j-l,v}}I\{Y_{l,v_i}+Z_{j-l,v}>\ln{m}-\ln{c} \}\Big|\mathscr G_{l}\right).
\end{multline*}
Then by  similar calculations as in (\ref{1}),
\begin{align}\label{F1}F_1&=e^{-\frac{m+1}{m}(\ln{m}-\ln{c})}\frac{m}{m+1}\sum_{i=1}^{b^l}\int_{0}^{\ln{m}-\ln{c}-Y_{l,{v_i}}}
 \sum_{j=l+1}^{L} b^{j-l}e^{\frac{t}{m}}d\mathbf{P}(Z_{j-l,{v}}\leq t)
\nonumber 
\\&=\sum_{i=1}^{b^l}\mu ^{-1}\prod_{r=1}^{l}W_{r,v_i} +o(1)=\mu ^{-1}+o(1).\end{align}
By similar calculations as in (\ref{1})--(\ref{ekvtill2}), we obtain
\begin{align}\label{F2}F_2&=\sum_{i=1}^{b^l}
\prod_{r=1}^{l}W_{r,v_i}\bigg((L-l)-\nonumber 
\\&~~~~~~~~\int_{0}^{\ln{m}-\ln{c}-Y_{l,{v_i}}}\frac{m}{m+1} \sum_{j=l+1}^{L} b^{j-l} 
e^{-t}d\mathbf{P}(Z_{j-l,{v}}\leq t)\bigg) +o(1)
\nonumber \\&=\sum_{i=1}^{b^l}\prod_{r=1}^{l}W_{r,v_i}\left( L-l-\mu ^{-1}\ln{m}+\mu ^{-1}\ln{c}-\frac{\sigma^2-\mu^2}{2\mu^2}-\mu ^{-1}\sum_{r=1}^{l}\ln{W_{r,v_i}}\right)+o(1)\nonumber 
\\&= L-l-\mu ^{-1}\ln{m}+\mu ^{-1}\ln{c}-\frac{\sigma^2-\mu^2}{2\mu^2}-\sum_{i=1}^{b^l}\mu ^{-1}\prod_{r=1}^{l}W_{r,v_i}\sum_{r=1}^{l}\ln{W_{r,v_i}}+o(1).
\end{align}
Thus, by applying the approximations of $F_1$ in (\ref{F1}) and $F_2$ in (\ref{F2}) we get
\begin{multline}\label{condexp6} \mathbf{E}\Big(\sum_{l \leq d(v)\leq L}\mathbf{E}\big(\widehat{\xi_v}\mathbf{1}
 [ \widehat{\xi_v}\leq c ]\big|\mathscr G_{L} \big)\Big|\mathscr G_{l}\Big)=\mu ^{-1}+L-l-\mu ^{-1}\ln{m}+\mu ^{-1}\ln{c}\\-\frac{\sigma^2-\mu^2}{2\mu^2}-\sum_{i=1}^{b^l}\mu ^{-1}\prod_{r=1}^{l}W_{r,v_i}\sum_{r=1}^{l}\ln{W_{r,v_i}}+o(1).
 \end{multline}

 Let $v_i$ be a vertex at depth $l$ and let $v$ be a vertex at depth $L$. Similarly as in (\ref{exptot1,1}) and (\ref{exptot2}) 
(compare with \cite[equations (78)--(79)]{holmgren}), we get that
\begin{align}\label{exptot4}&\mathbf{E}\Big(\Phi_v|\mathscr G_{l}\Big)
\nonumber\\ &=\frac{n}{\mu ^{-1}\ln{n}}+\frac{n(L-l)}{\mu ^{-2}\ln^2{n}}
-\sum_{i=1}^{b^l}\frac{n\prod_{r=1}^{l}W_{r,v_i}\sum_{r=1}^{l}\ln W_{r,v_i}}{\mu ^{-1}\ln^2{n}}+o(\frac{n}{\ln^2{n}}).\end{align}

From (\ref{condexp6}) and (\ref{exptot4}) we obtain that
$\mathbf{Var}\Big(\mathbf{E}(S_2|\mathscr G_{l})\Big)$ is o(1),
which shows (\ref{lightout1}).

For (\ref{scar1}) we proceed with the same method as for showing (\ref{tjorn}) and (\ref{lightout1}).
Recall from (\ref{S3}) that 
\begin{align*}
 S_3&=\sum_{{l\leq d(v)\leq L}}\mathbf{Var}\big(\widehat{\xi_v}\mathbf{1} [ \widehat{\xi_v}\leq c ]\big|
\mathscr G_{L}\big).
\end{align*}

By similar calculations as in (\ref{array6}) and (\ref{V1}) we get
\begin{align*}&\mathbf{E}\Big(\sum_{l\leq d(v)\leq L}\mathbf{Var}\big(\widehat{\xi_v}\mathbf{1} [ \widehat{\xi_v}\leq c ]\big|
\mathscr G_{L}\big)\Big|\mathscr G_{l}\Big)=I_1+I_2+o(1),
\end{align*} 
where
\begin{align*}I_1:&=e^{-\frac{2m+1}{m}(\ln{m}-\ln{c})}\sum_{i=1}^{b^l}\sum_{j=l+1}^{L}\frac{b^{j-l}m}{2}
e^{\frac{t}{m}}d\mathbf{P}(Z_{j-l,v}\leq t),\\I_2:&=\sum_{i=1}^{b^l}\mathbf{E}\Big(\sum_{j=l+1}^{L}\frac{b^{j-l}m}{2}\prod_{r=1}^{l}{W_{r,v_i}^2}\prod_{r=l+1}^{j}{W_{r,v}^2}
I\{Y_{l,v_i}+Z_{j-l,v}> \ln{m}-\ln{c} \}\Big|\mathscr G_{l}\Big).
\end{align*}

Using integration by parts  we calculate (similarly as in (\ref{pelle3}) and (\ref{pelle2}),
\begin{align*}  I_1+o(1)=I_2+o(1) =\sum_{i=1}^{b^l}
\prod_{r=1}^{l}{W_{r,v_i}} \frac{\mu^-1}{2}c+o(1)=\frac{\mu^-1}{2}c+o(1).
\end{align*}
Thus, $\mathbf{Var}\Big(\mathbf{E}(S_3\Big| \mathscr G_{l})\Big)$ is o(1), which shows (\ref{scar1}).

\end{proof}

\begin{proof}[Proof of Lemma \ref{lem8}]

Recall from (\ref{S1}) that \begin{align*}S_1&=\sum_{l\leq d(v)\leq L}
\mathbf{P}\big(\widehat{\xi_v}>x\big|\mathscr G_{L}\big).
\end{align*}

For showing (\ref{cia1}) we first note that
\begin{align*}&\mathbf{Var}\Big(\sum_{l\leq d(v)\leq L}
\mathbf{P}\big(\widehat{\xi_v}>x\big|\mathscr G_{L}\big)\Big|\mathscr G_{l}\Big) \nonumber \\&=
\sum_{\radsumma{l \leq d(v)\leq L,}{l \leq d(w)\leq L}}\mathbf{Cov}\Big(\mathbf{P}\big(\widehat{\xi_v}>x\big|\mathscr G_{L}\big),\mathbf{P}\big(\widehat{\xi_w}>x\big|\mathscr G_{L}\big)
\Big|\mathscr G_{l}\Big).
\end{align*}

To estimate these conditional covariances we can suppose that the closest ancestor $u$ for 
$v$ with $d(v)\leq L$, and $w$ with $d(w)\leq L$ is at depth $d\geq l$, since the other terms are just 0 because of independence.
For $d\geq l$, we use
\begin{align*}
&\mathbf{Cov}\Big(\mathbf{P}\big(\widehat{\xi_v}>x\big|\mathscr G_{L}\big),\mathbf{P}\big(\widehat{\xi_w}>x\big|\mathscr G_{L}\big)
\Big|\mathscr G_{l}\Big)\nonumber\\&\leq \mathbf{E}\Big(\mathbf{P}\big(\widehat{\xi_v}>x\big|\mathscr G_{L}\big)\mathbf{P}\big(\widehat{\xi_w}>x\big|\mathscr G_{L}\big)\Big|\mathscr G_{l}\Big),
\end{align*}
which implies
\begin{align}\label{fanny}&
 \mathbf{E}\Big( \mathbf{Cov}\Big(\mathbf{P}\big(\widehat{\xi_v}>x\big|\mathscr G_{L}\big),\mathbf{P}\big(\widehat{\xi_w}>x\big|\mathscr G_{L}\big)
\Big|\mathscr G_{l}\Big)\Big) \nonumber \\&\leq \mathbf{E}\Big(\mathbf{P}\big(\widehat{\xi_v}>x\big|\mathscr G_{L}\big)\mathbf{P}\big(\widehat{\xi_w}>x\big|\mathscr G_{L}\big)\Big).
\end{align}

 Denote by $(v_u,w_u)$ a general pair of vertices with closest ancestor $u$. Then  (\ref{fanny}) implies that
\begin{align}\label{array11}&\mathbf{E}\left(\mathbf{Var}\Big(S_1\Big|\mathscr G_{l}\Big)\right)\leq\nonumber\\&  \sum_{d=l}^{L}\sum_{u:d(u)=d}\sum_{(v_u,w_u)}\mathbf{E}\Big(\mathbf{P}\big(\widehat{\xi}_{v_u}>x\big|\mathscr G_{L}\big)\mathbf{P}\big(\widehat{\xi}_{w_u}>x\big|\mathscr G_{L}\big)\Big).
\end{align}

Recall that $\mathscr G_{d+1}$ is the $\sigma$-field generated by $\{\widehat{n_v}:d(v)\leq d+1\}$.
For the pair $(v_u,w_u)$ with  $d(u)=d$, conditioned on 
$\mathscr G_{d+1}$, $\mathbf{E}\left(\mathbf{P}\big(\widehat{\xi}_{v_u}>x\big|\mathscr G_{L}\big)\mid \mathscr G_{d+1}\right) $ 
and $\mathbf{E}\left(\mathbf{P}\big(\widehat{\xi}_{w_u}>x\big|\mathscr G_{L}\big) \mid \mathscr G_{d+1}\right)$ are independent. 
Thus,
\begin{align}\label{frankrike}
 &\mathbf{E}\Big(\mathbf{P}\big(\widehat{\xi}_{v_u}>x\big|\mathscr G_{L}\big)\mathbf{P}\big(\widehat{\xi}_{w_u}>x\big|\mathscr G_{L}\big)\Big| \mathscr G_{d+1}\Big)
 \nonumber\\&=\mathbf{E}\Big(\mathbf{P}\big(\widehat{\xi}_{v_u}>x\big|\mathscr G_{L}\big)\mid \mathscr G_{d+1}\Big)\cdot
\mathbf{E}\Big(\mathbf{P}\big(\widehat{\xi}_{w_u}>x\big|\mathscr G_{L}\big) \Big| \mathscr G_{d+1}\Big).
\end{align}

Let $v,w:v\wedge w=u$ denote that the vertices $v,w$ have closest ancestor $u$. 
Using (\ref{sup}) and (\ref{sup2}), by similar calculations as 
in (\ref{meas2}) for $v$, $w$ and $u$ with $d(v)=j$, $d(w)=k$ and $d(u)=d$ respectively (where $u$ is the closest ancestor 
to $v$ and $w$), we get 
 \begin{align}\label{array11,2}&
\sum_{d=l}^{L}\sum_{u:d(u)=d}\sum_{\radsumma{v,w\in T_u:}{v\wedge w=u}}\mathbf{E}\Big(\mathbf{P}\big(\widehat{\xi_{v}}>x\big|\mathscr G_{L}\big)\mathbf{P}\big(\widehat{\xi_{w}}>x\big|\mathscr G_{L}\big)\Big| \mathscr G_{d+1}\Big)
\\&\leq
\sum_{d=l}^{L}\sum_{u:d(u)=d}\sum_{j=d}^{L}\sum_{\radsumma{v\in T_u:}{d(v)=j}}
\mathbf{E}\Big(\frac{1}{m}\ln_+\frac{m\prod_{r=1}^{j} 
W_{r,v}}{x}\Big| \mathscr G_{d+1}\Big)\cdot\nonumber \\~~~~~~~~~~~~~~&\sum_{d=l}^{L}\sum_{u:d(u)=d}\sum_{k=d}^{L}\sum_{\radsumma{w\in T_u:}{d(w)=k}}\mathbf{E}\Big(\frac{1}{m}\ln_+\frac{m\prod_{r=1}^{k} W_{r,w}}{x}\Big| \mathscr G_{d+1}\Big),
\end{align}
where the inequality follows by applying (\ref{frankrike}) 
and using analogous calculations as in \cite[equations (83)--(84)]{holmgren}.
 Note that the expected value of the left hand-side of the inequality in (\ref{array11,2}) is equal to the right 
hand-side of the inequality in (\ref{array11}).
Let $u$ be the closest ancestor vertex of $v$ and $w$. Let $u_v$ be the child of $u$ that is an ancestor of $v$, respectively $u_w$ be the child of $u$ that is an ancestor of $w$.
Let $\widehat{\mathcal{W}}_{u,v}$ be the component in the split vector of vertex $u$ that corresponds to the child $u_v$ of $u$, and use the analogous notation for $\widehat{\mathcal{W}}_{u,w}$.
For a triple $(v,w,u)$ with $d(v)=j$, $d(w)=k$ and $d(u)=d$ we have  \begin{align}\label{widehat}\widehat{n_v}&=n\prod_{r=1}^{j}W_{r,v}=n\widehat{\mathcal{W}}_{u,v}\prod_{r=1}^dW_{r,u}\prod_{r=d+2}^{j} W_{r,v},
\nonumber\\\widehat{n_w}&=n\prod_{r=1}^{k}W_{r,w}=n\widehat{\mathcal{W}}_{u,w}\prod_{r=1}^dW_{r,u}\prod_{r=d+2}^{k} 
W_{r,w}.\end{align} 

For given $d(u)=d\geq l$, $d(v)=j$ and $d(w)=k$, there are at most $b^d$ choices of $u$, 
and then at most $b^{j-d}$ choices of $v$ and $b^{k-d}$ choices of $w$. 
(We can assume that $j>d+1$ and $k>d+1$, since it is easy to see that the other terms are few and the sum of them is small.)
For the child $u_v$ of $u$, $d(u_v)=d+1$  and we have  
$Y_{d+1,u_v}=-\sum_{r=1}^d\ln{W_{r,u}}-\ln{\widehat{\mathcal{W}}_{u,v}}$ (and for the child $u_w$ of $u$, $Y_{d+1,u_w}$ is defined in analogy). 
 Recall the definition of $Z_{j-l,v}:=Y_{j,v}-Y_{d+1,u_v}$ in (\ref{ny}).
 For the vertex $v$ with $d(v)=j$ we have that $Z_{j-d-1,v}:=-\sum_{r=d+2}^{j}\ln{W_{r,v}}$ 
(and the analogous notation for $Z_{k-d-1,v}$). (For simplicity we skip the vertex index in the calculations below.) 
Thus, by similar calculations as in (\ref{meas3}) and (\ref{meas3,2})  the sum in (\ref{array11,2}) is equal to
\begin{align*}
&\frac{1}{m^2}O\Bigg(\sum_{d= l}^{L} \sum_{d(u)=d}\Big{(} \sum_{j= d+2}^{L} 
\int_{0}^{\ln{m}-\ln{x}-Y_{d+1,v}} b^{j-d-1} \mathbf{P}(Z_{j-d-1}\leq t)\Big)\cdot \nonumber\\&
\Big(\sum_{k= d+2}^{L} \int_{0}^{\ln{m}-\ln{x}-{Y}_{d+1,w}} b^{k-d-1}\mathbf{P}(Z_{k-d-1}\leq t)  \Big)
\Bigg)=
\mathcal O \Bigg(\sum_{d=l}^{L}\sum_{d(u)=d}\frac{1}{x^2}\prod_{j=1}^dW_{j}^{2}\Bigg).
\end{align*}

Since $\mathbf{E}(W_{j}^{2})<\frac{1}{b}$ the expected value of this is  $o(1)$,
and thus the right hand-side of the inequality in (\ref{array11}) is $o(1)$. Hence,
$\mathbf{E}\Big(\mathbf{Var}(S_1|\mathscr G_{l})\Big)$ is o(1),
which shows (\ref{cia1}). We proceed by showing  (\ref{cia2}). 
 Recall from (\ref{S2}) that \begin{align*}
             S_2&=\sum_{{l\leq d(v)\leq L}}\mathbf{E}\big(\widehat{\xi_v}\mathbf{1} 
[ \widehat{\xi_v}\leq c ]\big|\mathscr G_{L} \big)- \frac{\mu^{-2}\ln^2{n}}{n}\cdot \Phi_v,
            \end{align*}
where \begin{align*}&\Phi_v=\frac{n}{\mu^{-1}\ln{n}}
-\sum_{d(v)=L}\frac{n\prod_{r=1}^{L}W_{r,v}\sum_{r=1}^{L}
\ln{W_{r,v}}}{\mu^{-1}\ln^2{n}}.
\end{align*}
First we consider 
\begin{align}\label{tarray2}&\mathbf{Var}\Big(\sum_{l\leq d(v)\leq L}\mathbf{E}\big(\widehat{\xi_v}\mathbf{1} 
[ \widehat{\xi_v}\leq c ]\big|\mathscr G_{L} \big)
\Big|\mathscr G_{l}\Big)\nonumber \\&=
\sum_{\radsumma{l\leq d(v)\leq L,}{l\leq d(w)\leq L}}\mathbf{Cov}\Big(\mathbf{E}\big(\widehat{\xi_v}\mathbf{1} 
[ \widehat{\xi_v}\leq c ]\big|\mathscr G_{L} \big),\mathbf{E}\big(\widehat{\xi_w}\mathbf{1} 
[ \widehat{\xi_w}\leq c ]\big|\mathscr G_{L} \big)
\Big|\mathscr G_{l}\Big).
\end{align}
 As we argued for showing (\ref{cia1}), we can suppose that the closest ancestor $u$ for $v$ and $w$ is at depth $d\geq l$.
Similar to (\ref{fanny}),
\begin{align*}
&\mathbf{E}\Big(\mathbf{Cov}\Big(\mathbf{E}\big(\widehat{\xi_v}\mathbf{1} 
[ \widehat{\xi_v}\leq c ]\big|\mathscr G_{L} \big),\mathbf{E}\big(\widehat{\xi_w}\mathbf{1} 
[ \widehat{\xi_w}\leq c ]\big|\mathscr G_{L} \big)
\Big|\mathscr G_{l}\Big)\Big)\nonumber\\&\leq \mathbf{E}\Big(\mathbf{E}\big(\widehat{\xi_v}\mathbf{1} 
[ \widehat{\xi_v}\leq c ]\big|\mathscr G_{L} \big)\mathbf{E}\big(\widehat{\xi_w}\mathbf{1} 
[ \widehat{\xi_w}\leq c ]\big|\mathscr G_{L} \big)\Big).
\end{align*}

For a vertex $v$ with $d(v)=j$,
\begin{align*}\mathbf{E}\big(\widehat{\xi_v}\mathbf{1} 
[ \widehat{\xi_v}\leq c ]\big|\mathscr G_{L} \big)=\frac{m\widehat{n_v}}{n(m+1)}e^{-\frac{m+1}{m}\ln_+({\frac{m\widehat{n_v}}{nc}})}\leq \frac{\widehat{n_v}}{n}
=\prod _{r=1}^{j}W_{r,v}.\end{align*}

Denote by $(v_u,w_u)$ a pair of vertices with closest ancestor $u$ as in  (\ref{array11}).
Consider one such pair $(v_u,w_u)$, and let $d(u)=d$, $d(v)=j$ and $d(w)=k$. 
Since 
$\mathbf{E}(W_{r,v}^{2})<\frac{1}{b^{1+\delta}}$ for some $\delta>0$ it follows that
\begin{align}\label{covexp2}\mathbf{E}\Big(\mathbf{E}\big(\widehat{\xi}_{v_u}\mathbf{1} 
[ \widehat{\xi}_{v_u}\leq c ]\big|\mathscr G_{L} \big)\mathbf{E}\big(\widehat{\xi}_{w_u}\mathbf{1} 
[ \widehat{\xi}_{w_u}\leq c ]\big|\mathscr G_{L} \big)\Big)&\leq C_1\mathbf{E}\prod _{r=1}^dW_{r,u}^{2}b^{-(j-d)-(k-d)}\nonumber\\&
\leq C_1\Big(\frac{1}{b^{1+\delta}}\Big)^db^{-(j-d)-(k-d)},
\end{align}
where $C_1$ is a constant depending on $\mathbf{E}(\widehat{W}_{u,v}\widehat{W}_{u,w})$, where $\widehat{W}_{u,v}$ and 
$\widehat{W}_{u,w}$ are the random variables that we introduced for (\ref{widehat}).
 Thus, by using (\ref{tarray2})--(\ref{covexp2}), and as in (\ref{array11,2}) letting $v,w:v\wedge w=u$ denote that the vertices $v,w$ have closest ancestor $u$  we get
\begin{align}\label{scarlet1}&\mathbf{E}\bigg(\mathbf{Var}\Big(\sum_{l \leq d(v)\leq L}
\mathbf{E}\big(\widehat{\xi_v}\mathbf{1} 
[ \widehat{\xi_v}\leq c ]\big|\mathscr G_{L} \big)\Big|\mathscr G_{l}
\Big)\bigg)\nonumber\\&\leq
\sum_{d=l}^{L}\sum_{u:d(u)=d}\sum_{\radsumma{v,w\in T_u:}{v\wedge w=u}}\mathbf{E}\Big(\mathbf{E}\big(\widehat{\xi_v}\mathbf{1} 
[ \widehat{\xi_v}\leq c ]\big|\mathscr G_{L} \big)\mathbf{E}\big(\widehat{\xi_w}\mathbf{1} 
[ \widehat{\xi_w}\leq c ]\big|\mathscr G_{L} \big)\Big)
\nonumber\\& \leq C_1 \sum_{d=l}^{L}b^{-\delta d}\sum_{j=d}^{L}b^{j-d-(j-d)}
\sum_{k=d}^{L}b^{k-d-(k-d)}\leq C_2 L^2b^{-\delta l} \rightarrow 0,
\end{align}
 where $C_2$ is a constant. (Compare with the calculations in \cite[equation (87)]{holmgren}.)
 We now show that \begin{align}
\label{covexp3}\mathbf{E}\bigg(\mathbf{Var}\Big( \frac{\mu^{-2}\ln^2{n}}{n}\cdot\Phi_v\Big|\mathscr G_{l}\Big)\bigg)\rightarrow 0. \end{align}
To show this, it is enough to show that
\begin{align*}\mathbf{E}\bigg(\mathbf{Var}\Big(\sum_{v:d(v)=L}\prod_{r=1}^{L}W_{r,v}\sum_{r=1}^{L}\ln{W_{r,v}}\Big|\mathscr G_{l}\Big)\bigg) \rightarrow 0.\end{align*}
Using (\ref{scarlet1}), we obtain for each $s\leq L$, 
\begin{align*}&\mathbf{E}\bigg(\mathbf{Var}\Big(\sum_{v:d(v)=L}\prod_{r=1}^{L}W_{r,v}\ln{W_{s,v}}\Big|\mathscr G_{l}\Big)\bigg)= \mathcal{O}\left(L^2b^{-\delta l}\right).\end{align*}
Thus, the conditional H\"{o}lder inequality, see e.g., \cite[p. 476]{gut}, yields
(\ref{covexp3}).
From (\ref{scarlet1}) and (\ref{covexp3}) and again applying the conditional H\"{o}lder 
inequality we deduce that
$\mathbf{E}\Big(\mathbf{Var}(S_2|\mathscr G_{l})\Big)$ is o(1),
which shows (\ref{cia2}). 

Recall from (\ref{S3}) that 
\begin{align*}
 S_3&=\sum_{{l\leq d(v)\leq L}}\mathbf{Var}\big(\widehat{\xi_v}\mathbf{1} [ \widehat{\xi_v}\leq c ]|
\mathscr G_{L}\big).
\end{align*}
It remains to show that $\mathbf{E}\Big(\mathbf{Var}(S_3|\mathscr G_{l})\Big)$ is o(1). 
To show this we observe that
\begin{align*}&\mathbf{Var}\big(\widehat{\xi_v}\mathbf{1} [ \widehat{\xi_v}\leq c ]\big|
\mathscr G_{L}\big)\leq
\mathbf{E}\big(\widehat{\xi_v}^2\mathbf{1} [ \widehat{\xi_v}\leq c ]\big|\mathscr G_{L}\big)\leq c
\mathbf{E}\big(\widehat{\xi_v}\mathbf{1} [ \widehat{\xi_v}\leq c ]\big|
\mathscr G_{L}\big),
\end{align*}
and thus (\ref{scarlet2}) follows from (\ref{covexp2}) by similar calculations as in (\ref{scarlet1}).

 \end{proof}

$\mathbf{Acknowledgement}:$

I gratefully acknowledge the help and support of Professor Svante Janson, for introducing me to this 
problem area and for helpful discussions and guidance.

\end{document}